\documentclass[11pt, a4paper]{article}
\usepackage{amssymb,amsfonts,amsmath,amsthm}
\usepackage{adjustbox}
\newcommand{\added}[1]{{#1}}

\usepackage{graphicx}
\usepackage{enumerate}
\usepackage[round]{natbib}
\usepackage{url} 
\usepackage{tikz}
\usetikzlibrary{matrix}
\usepackage{subcaption}
\usepackage{amsmath}
\usepackage{bm}
\definecolor{newcolor}{RGB}{0,0,180} 

\newtheorem{lem}{Lemma}

\newtheorem{prp}{Proposition}
\newtheorem{rmk}{Remark}
\newtheorem{thm}{Theorem}

\DeclareMathOperator*{\argmin}{arg\,min}

\DeclareMathOperator*{\essinf}{ess\,inf}
\newcommand{\eqdef}{:=}
\DeclareMathOperator*{\LR}{LR}

\begin{document}
	

	
    \title{\bf Anytime Validity is Free: Inducing Sequential Tests}
    \author{Nick W. Koning\footnote{Corresponding author n.w.koning@ese.eur.nl.}\ \& Sam van Meer\footnote{Both authors contributed equally to this work.}\hspace{.2cm}\\
    Econometric Institute, Erasmus University Rotterdam, the Netherlands}
    \maketitle
    \vspace{-.5cm}
    \begin{abstract}
        Anytime valid sequential tests permit us to stop testing based on the current data, without invalidating the inference.
        Given a maximum number of observations $N$, one may believe this must come at the cost of power when compared to a conventional test that waits until all $N$ observations have arrived.
        Our first contribution is to show that this is false: for any valid test based on $N$ observations, we \added{show how to construct} an anytime valid sequential test that matches it after $N$ observations.
        \added{Our second contribution is that we may continue testing by using the outcome of a $[0, 1]$-valued test as a conditional significance level in subsequent testing}, leading to an overall procedure that is valid at the original significance level.
        This shows \added{that} anytime validity and optional continuation are readily available in traditional testing, without requiring explicit use of e-values.
        We illustrate this by deriving the anytime valid sequentialized $z$-test and $t$-test, which at time $N$ coincide with the traditional $z$-test and $t$-test.
        \added{Finally, we characterize the SPRT by invariance under test induction, and also show under an i.i.d. assumption that the SPRT is induced by the Neyman-Pearson test for a tiny significance level and huge $N$.}
    \end{abstract}
    
	\noindent%
	\vspace{.5cm}

    \section{Introduction}
        Suppose that we are to observe $N$ observations $X^N := (X_1, \dots, X_N)$.
        We are interested in testing the null hypothesis that the data is sampled from the distribution $\mathbb{P}$, against the alternative hypothesis that it is sampled from the distribution $\mathbb{Q}$.
        Traditionally, we wait until all $N$ observations have been collected, and then perform a test which either rejects the hypothesis or not.

        It is common to \added{define} such a test as a random variable $\phi_N$ on the interval $[0, 1]$, that depends on the data $X^N$.
        Here, $\phi_N = 1$ represents a rejection of the hypothesis and $\phi_N = 0$ a non-rejection.
        If $0 < \phi_N < 1$, the value of $\phi_N$ is traditionally interpreted as the probability with which we may subsequently reject by using external randomization.
        However, we have no intention to randomize, and it will suffice, for now, to view a test outcome in $(0, 1)$ as `progress' towards a rejection.

        It is standard practice to use a test that is valid at some level of significance $\alpha > 0$.
        This means that the probability that it rejects the null hypothesis is at most $\alpha$ if the null hypothesis is true.
        This can be translated to a condition on the expectation of $\phi_N$:
        \begin{align}\label{eq:type-1_intro}
            \mathbb{E}^{\mathbb{P}}[\phi_N] 
                \leq \alpha,
        \end{align}
        for every data-independent $N$.

        \added{An unfortunate feature of traditional testing is that we must sit on our hands and wait until all $N$ observations have arrived to preserve this validity.}
        One may naively believe that, after seeing $n < N$ observations, we could simply use the test $\phi_n$ that we would have used if we had set out to collect $n$ observations.
        However, this makes the number of observations data-dependent, so that the resulting procedure is generally not valid.
        For this to be allowed, a sequence of tests $\phi_1, \phi_2, \dots$ should not just be valid, but \emph{anytime valid}:
        \begin{align}
            \mathbb{E}^\mathbb{P}[\phi_\tau] \leq \alpha,
        \end{align}
        for every data-dependent (`stopping') time $\tau$.

        Following seminal work by Robbins, Darling, Wald and others in the previous century, there has recently been a renaissance in anytime valid testing \citep{howard2021time, shafer2021testing, ramdas2023game, grunwald2024safe}.
        Such anytime valid sequential tests are typically of a different form than traditional tests.
        The most popular sequential test remains the Sequential Probability Ratio Test (SPRT), which equals $\alpha\textnormal{LR}_n \wedge 1$ at every observation $n$, where $\textnormal{LR}_n$ denotes the likelihood ratio between $\mathbb{Q}$ and $\mathbb{P}$ of the $n$ observations.\footnote{The name SPRT is usually reserved for the `binary' version of this test: $\mathbb{I}\{\alpha\textnormal{LR}_n \geq 1\} \leq \alpha\textnormal{LR}_n \wedge 1$, and can be slightly further improved \citep{fischer2024improving, koning2024continuoustesting}.}
        
        At time $N$, this sequential test is usually substantially less powerful than the optimal `Neyman-Pearson' test.
        For example, in a standard normal location setting where we test $\mu = 0$ against $\mu = .3$ with $N = 100$ observations, the $z$-test rejects with probability 91\% while this sequential test rejects at any time $n \leq N$ with probability just 79\%.\footnote{This increases to 84\% using external randomization at the final observation $N$.}
        Moreover, this power comparison is very generous towards the sequential test, as it uses oracle knowledge of the alternative $\mu$, whereas the $z$-test does not.
        If we were to sequentially learn the alternative with the maximum likelihood estimator, then its power is just 47\%.

        \subsection{Inducing a sequential test}
            The large power gap between the state-of-the-art sequential test and the traditional Neyman-Pearson test seems to be generally viewed as the cost of anytime validity.
            Indeed, this additional anytime validity must surely come at a price?
        
            Our first contribution is to show that \emph{this is false}: anytime validity can be obtained for free.
            In particular, for any valid test $\phi_N$ we show how to construct a sequence of tests $\phi_1', \dots \phi_N'$ that is anytime valid and matches $\phi_N$ at the end: $\phi_N' = \phi_N$.

            Moreover, the sequential test is of a simple form: it simply equals the conditional expectation of $\phi_N$ given the current data $X^n$ under the null hypothesis $\mathbb{P}$:
            \begin{align*}
                \phi_n'
                    := \mathbb{E}^{\mathbb{P}}[\phi_N \mid X^n],
            \end{align*}
            for all $n \geq 0$.
            This may be interpreted as tracking the conditional probability that $\phi_N$ will reject, including randomization at the final step to come to a binary decision if necessary.
            We illustrate this for the one-sided $z$-test and $t$-test in  Figure~\ref{fig:intro} for a typical sample of i.i.d. normal data with a power of 91\% at $N = 100$ observations.

            \begin{figure}
                \centering
                \includegraphics[width=9.5cm]{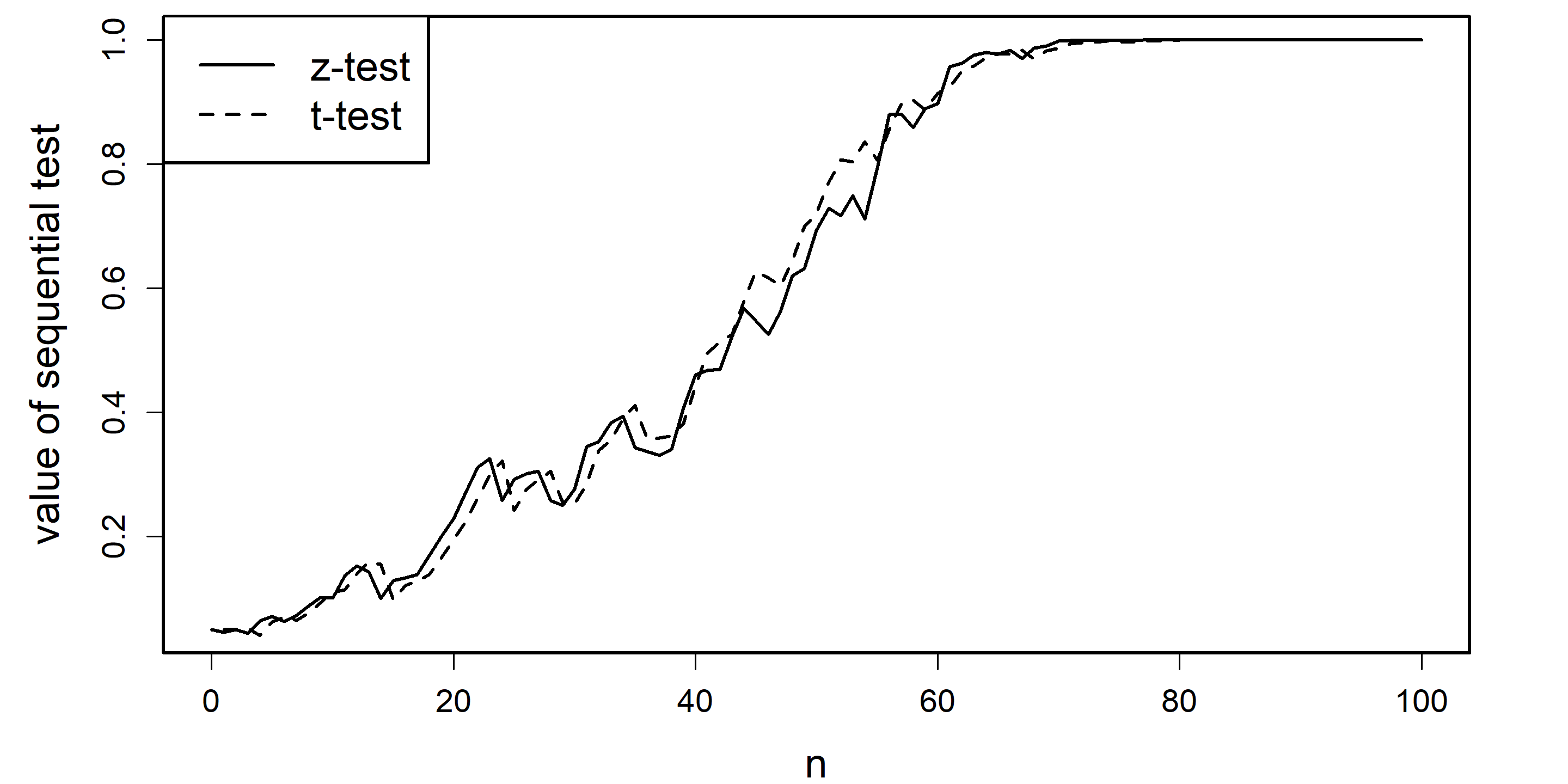}
                \caption{Illustration of the sequential $z$-test and $t$-test over $100$ observations sampled from $\mathcal{N}(.3, 1)$.}
                \label{fig:intro}
            \end{figure}

            We formalize the induction of sequential tests in Section \ref{sec:inducing_sequential_test}.
            There, we also show that every admissible anytime valid sequential test can be induced from some `terminal test'.
            Moreover, we generalize to composite hypotheses and study composite hypotheses that possess the Monotone Likelihood Ratio property.
            In Section \ref{sec:ztest} and Section \ref{sec:t-test}, we apply this to the $z$-test and $t$-test, respectively.

        \subsection{Using test outcomes as conditional significance levels}
            A natural next question is \emph{what we should do with a test outcome $\phi_n \in (0, 1)$}: are such outcomes just entertainment while we wait for the eventual rejection decision of $\phi_N$?
            This is important, because even though $\phi_n' > 0.999$ for $n > 77$ in Figure~\ref{fig:intro}, we technically have that $\phi_n' \in (0, 1)$ for all $n < N$.
            Indeed, in this example the final observation may be arbitrarily small or large with some tiny probability, swinging $\phi_N$ to a rejection or non-rejection.

            Our second contribution is to show the merit of test outcomes in $(0, 1)$ through \emph{optional continuation}: the outcome of a $(0, 1)$-valued test may be used as a \added{\emph{conditional significance level}} in subsequent testing.
            In particular, given the outcome of a test $\phi_1$ that is (unconditionally) valid at significance level $\alpha$:
            \begin{align*}
                \mathbb{E}^{\mathbb{P}}[\phi_1] \leq \alpha,
            \end{align*}
            we may initialize a new test $\phi_{2}$ that is conditionally valid at conditional significance level $\phi_1$:
            \begin{align*}
                \mathbb{E}^{\mathbb{P}}[\phi_{2} \mid \phi_1] \leq \phi_1.
            \end{align*}
            The new test $\phi_{2}$ is then (unconditionally) valid at (unconditional) significance level $\alpha$:
            \begin{align*}
                \mathbb{E}^{\mathbb{P}}[\phi_{2}]
                    = \mathbb{E}^{\mathbb{P}}[\mathbb{E}^{\mathbb{P}}[\phi_{2} \mid \phi_1]]
                    \leq \mathbb{E}^{\mathbb{P}}[\phi_1]
                    \leq \alpha.
            \end{align*}
            Overall, this supports the idea that an outcome $\phi_1 \in (0, 1)$ can be interpreted as `progress towards a rejection'.
            
            This idea also gives a new interpretation to tests on $[0, 1]$, beyond the traditional randomized testing interpretation.
            In fact, randomized testing may be viewed as a special case, where a $[0, 1]$-valued test $\phi$ is implicitly followed by a test at conditional level $\phi$ that uses uninformative data and so `randomly' rejects with probability $\phi$.

            We formalize these ideas in Section \ref{sec:test_as_significance_level}.
            There, we show how this procedure may be used to modify a sequential experiment at a stopping time $\tau$.
            Moreover, we generalize the procedure to composite hypotheses, and discuss adaptively switching to a different baseline significance level, which allows one to continue after a rejection at level $\alpha$ in order to obtain a rejection at some level $\alpha' < \alpha$.
            Furthermore, we discuss the relationship to optional continuation with e-values.

        \subsection{Contributions to the anytime validity literature}
            A more abstract contribution of our work is the establishment of a bridge between classical testing and the recent literature on anytime validity and optional continuation.
            In the current literature, anytime validity and optional continuation are often presented as the primary advantages of e-values \citep{ramdas2023game, grunwald2024safe}.
            Indeed, inference based on e-values is frequently described as a different paradigm from inference based on tests and p-values.
            Anytime validity is then viewed as a natural property of the e-value paradigm, which is not easily available in the traditional paradigm.
            Our work dispels this myth: both anytime validity and optional continuation come naturally to tests, and it is not necessary to (explicitly) introduce e-values.

            Our work also adds a new tool to the toolbox of anytime valid sequential testing: the induction of a sequential test from a terminal test.
            This effectively constructs a sequential test in a `backward' fashion by relying on a \emph{Doob martingale}.
            This is different from the `forward' construction that is common in the literature, which relies on designing martingale-increments and sequentially accumulating these.
            To the best of our knowledge, we are the first to propose such a Doob martingale approach to actively construct anytime valid sequential tests.
            Doob martingales have been featured in the context of anytime valid sequential testing before, but only as a technical tool, such as in a proof in Section 6.3 of \citet{ramdas2022admissible} and in the construction of time uniform concentration inequalities in \citet{howard2020time}.

            In Section \ref{sec:e-values} we describe the precise connection of our work to the e-value literature.
            There, we show that the popular log-utility-optimal e-process (SPRT) is characterized by an invariance property under the operation of inducing a sequential test.
            To the best of our knowledge, this is the first motivation of this process that does not rely on an independence assumption over time.
            Moreover, we discover that the log-utility-optimal e-process (SPRT) can be induced from a limiting Neyman-Pearson test with $\alpha_N \to 0$ as $N \to \infty$, substantially generalizing a finding of \citet{breiman1961optimal}.
            Finally, our study of optional continuation for a composite hypothesis $H$ in Section \ref{sec:pointwise_continuation} suggests that for the purpose of optional continuation it is better to report a test (e-value) for every $\mathbb{P} \in H$ rather than a single test (e-value) for $H$.
            This echoes recent calls to report this collection, under the name \emph{fuzzy confidence set} \citep{koning2025optimal} or \emph{E-posterior} \citep{grunwald2023posterior}.
        
        \subsection{Other related literature}
            \citet{koning2024continuoustesting} unifies e-values and tests, arguing that e-values are merely tests disguised by rescaling from $[0, 1]$ to $[0, 1/\alpha]$, so that calls to directly use an e-value as evidence are equivalent to calls to directly report a test outcome in $[0, 1]$.
            Our work supports this call, by adding an additional interpretation to such test outcomes as a conditional significance level for subsequent experimentation.
            
            We are not the first to consider tracking the conditional probability that a test will reject.
            Indeed, it also appears in a stream of literature initiated by \citet{GordonLan01011982}, who propose to reject whenever $\phi_n' > \gamma$ for some pre-specified $\gamma$ under the name \emph{stochastic curtailment}.
            This induces a sequential test $\phi_n^{\gamma} = \mathbb{I}\{\phi_n' > \gamma\}$, which they show has a Type I error bounded by $\alpha / \gamma$.
            In subsequent literature, the induced test $\phi_n'$ appears under the name \emph{conditional power}, and is also used as a diagnostic tool to stop for futility \citep{lachin2005review}.
            The idea to use $\mathbb{E}[\phi_N \mid X^n]$ as a significance level appears in \citet{muller2004general}.
            Perhaps our key conceptual leap is that $\mathbb{E}[\phi_N \mid X^n]$ itself can be replaced by any $[0, 1]$-valued test, showing the general evidential merit of such $[0, 1]$-valued tests.
            Moreover, by applying the rich mathematical toolbox of martingales, filtrations and stopping times we can expand beyond this line of work.
            To the best of our knowledge, this stream of literature has gone entirely unnoticed in the current renaissance in anytime valid testing, and our work shows the relationship between the streams of literature.

    \section{Inducing a sequential test}\label{sec:inducing_sequential_test}
        \subsection{Setup}
            Let $\mathcal{X}$ be our sample space equipped with some sigma-algebra $\mathcal{F}$ that encapsulates all the possible information that can be obtained.
            All processes and stopping times are silently considered to be measurable with respect to a filtration $(\mathcal{F}_n)_{n \geq 0}$, where $\mathcal{F}_n \subseteq \mathcal{F}$ describes the available information at time $n \geq 0$.
            Statements such as `$\mathcal{F}_n$-measurable' may be interpreted as `known when the information $\mathcal{F}_n$ is revealed'.
            Throughout, we abbreviate the notation for sequences $(\mathcal{F}_n)_{n \geq 0}$ to $(\mathcal{F}_n)$.
            For simplicity, let $\mathcal{F}_0 = \{\emptyset, \mathcal{X}\}$ represent the information set before any data has been observed, so that we can write the expectation $\mathbb{E}[\,\cdot\,]$ for the conditional expectation $\mathbb{E}[\,\cdot \mid \mathcal{F}_0]$.
            \added{Furthermore, we define $\mathcal{F}_\infty = \sigma(\bigcup_{n \geq 0} \mathcal{F}_n)$ and note that $\mathcal{F}_\infty \subseteq \mathcal{F}$.}
            
            We define a hypothesis $H$ as a collection of probability measures $\mathbb{P} \in H$ on our sample space.
            We define a test as an \added{$\mathcal{F}$-measurable} map $\phi : \mathcal{X} \to [0, 1]$.
            We fix a level $\alpha > 0$ throughout, and say that the test $\phi$ is valid for $H$ if
            \begin{align}
                \sup_{\mathbb{P} \in H} \mathbb{E}^{\mathbb{P}}[\phi] \leq \alpha.
            \end{align}
            A sequence of tests $(\phi_n)$ adapted to the filtration $(\mathcal{F}_n)$ is said to be anytime valid for $H$ if the test $\phi_\tau$ is valid for every (possibly infinite) stopping time $\tau$ adapted to $(\mathcal{F}_n)$:  
            \begin{align}
                \sup_{\mathbb{P} \in H} \mathbb{E}^{\mathbb{P}}[\phi_\tau] \leq \alpha.
            \end{align}
            \added{Here, we define $\phi_\infty := \limsup_{n \to \infty} \phi_n$, which is $\mathcal{F}_\infty$-measurable.}
            
        \subsection{Simple hypotheses}
            In this section, we study the induction of a sequential test $(\phi_n)$ from a test $\phi$ for a simple hypothesis $H = \{\mathbb{P}\}$.
            Given an $\mathcal{F}$-measurable test $\phi$, we define its (induced) sequential test $(\phi_n)$ as
            \begin{align}
                \phi_n
                    := \mathbb{E}^{\mathbb{P}}[\phi \mid \mathcal{F}_{n}],
            \end{align}
            which is $\mathbb{P}$-almost surely uniquely defined.

            In Theorem \ref{thm:simple}, we discuss the properties that a test $\phi$ passes on to its sequential test $(\phi_n)$.
            The most important result here is that if $\phi$ is valid, then $(\phi_n)$ is anytime valid.
            Interestingly, we also obtain a strong converse: if $\phi$ is not valid, then not only is $(\phi_n)$ not anytime valid; $\phi_\tau$ is not valid for \emph{any} stopping time.
            Furthermore, if $\phi$ is measured at some (stopping) time $\tau$, then the sequential test $(\phi_n)$ equals it at time $\tau$.
            By separating out the final claim, we also wish to highlight that we can induce a sequential test $(\phi_n)$ from a test $\phi$ even if we have no hope or no intention to actually observe it.

            The proof of this result and all other omitted proofs may be found in Appendix \ref{appn:proofs}.
            
            \begin{thm}[Simple hypotheses]\label{thm:simple}
                Let $\phi$ be an $\mathcal{F}$-measurable test.
                \begin{enumerate}
                    \item If $\phi$ is valid, then $(\phi_n)$ is anytime valid.
                    \item \added{If $\phi$ is not valid, then $\phi_\tau$ is not valid for any $\tau$.}
                    \item If $\phi$ is \added{$\mathcal{F}_\tau$-measurable}, then \added{$\phi_\tau = \phi$.}
                \end{enumerate}
            \end{thm}

            \added{
            Looking in the other direction, any sequential test $(\phi_n)$ induces an $\mathcal{F}_\infty$-measurable (and hence $\mathcal{F}$-measurable) \emph{terminal test} $\phi_\infty := \limsup_{n \to \infty} \phi_n$.
            In Proposition \ref{prp:validity_terminal_test}, we show that if $(\phi_n)$ is anytime valid, then $\phi_\infty$ is valid.

            \begin{prp}\label{prp:validity_terminal_test}
                If $(\phi_n)$ is anytime valid, then $\phi_\infty$ is valid.
            \end{prp}

            In Proposition \ref{prp:admissible}, we show that the construction we describe in Theorem \ref{thm:simple} is a powerful tool, in the sense that every anytime valid admissible sequential test may be induced from a test.
            Here, admissibility means that there exists no other anytime valid sequential test $(\phi_n')$ such that $\mathbb{P}(\phi_n' \geq \phi_n) = 1$ for all $n$ and $\mathbb{P}(\phi_n' > \phi_n) > 0$ for some $n$ \citep{ramdas2022admissible}.
            A consequence of this result is that all properties of an (admissible) sequential test are summarized by a single test $\phi$.
            
            \begin{prp}\label{prp:admissible}
                If $(\phi_n)$ is admissible, then there exists an $\mathcal{F}$-measurable test $\phi$ such that $\phi_n = \mathbb{E}^{\mathbb{P}}[\phi \mid \mathcal{F}_n]$, $\mathbb{P}$-almost surely, for every $n$.
            \end{prp}
            }

            \added{
            \begin{rmk}[Observing how a test comes to its conclusion]
                By tracking the induced sequential test $(\phi_n)$ of the test $\phi$, we are effectively observing how $\phi$ uses the information to come to its conclusion.
                Indeed, the starting point $\phi_0$ of the sequential test coincides with the size of the test $\phi$: $\phi_0 = \mathbb{E}^{\mathbb{P}}[\phi \mid \mathcal{F}_0] = \mathbb{E}^{\mathbb{P}}[\phi]$.
                Moreover, the value $\phi_n = \mathbb{E}^{\mathbb{P}}[\phi \mid \mathcal{F}_n]$ can be interpreted as reflecting the current expectation about the eventual outcome of the test $\phi$.
            \end{rmk}
            }
            
            \begin{rmk}[Doob martingale]
                In the literature on anytime valid inference, the standard strategy to build anytime valid sequential tests is to propose a sequence $(\phi_n)$ of tests and then show that this constitutes a martingale.
                Our approach here turns this around: we start with some target test $\phi$, from which we then induce a sequential test $(\phi_n)$ by conditioning on the filtration.
                For the simple hypothesis setting, this approach is also known as the Doob martingale of $\phi$.
            \end{rmk}

            \added{
            \begin{rmk}[Inducing test at stopping time]\label{rmk:inducing_at_stopping_time}
                In the introduction, we focused on inducing a sequential test from a test that is measurable at some data-independent time $N$.
                In Theorem \ref{thm:simple}, we show that we may also induce a sequential test from a test at some stopping time, so that the sequential test coincides with the test at the stopping time.
                As the stopping time is part of the sequential test that we induce, this stopping time must be chosen at the moment the sequential test is induced.
            \end{rmk}
            }
            \added{
            \begin{rmk}
                As closed forms for $\phi_n = \mathbb{E}[\phi \mid\mathcal{F}_n]$ are not always available, one may have to resort to numerical approximations.
                One potential approach is to represent $\phi_n$ as the outcome of an optimization problem, by using the representation of the conditional expectation of $\phi$ as its $L^2$-projection onto the class of $\mathcal{F}_n$-measurable functions: $\phi_n = \argmin_{\phi' : \mathcal{F}_n\textnormal{-measurable}} \mathbb{E}^{\mathbb{P}}\left[\left(\phi - \phi'\right)^2\right]$.
                Another is to sequentially update Monte Carlo approximations of $\mathbb{E}[\phi \mid \mathcal{F}_n]$, for example via particle filtering~\citep{DoucetDeFreitasGordon2001}.
            \end{rmk}
            }
            
        \subsection{Composite hypotheses}
            In this section, we generalize Theorem \ref{thm:simple} to composite null hypotheses $H$.
            In the composite setting, we rely on Proposition \ref{prp:e_to_essinf_e} which shows that a test $\phi^H$ for $H$ can be represented as the infimum $\essinf_{\mathbb{P} \in H}$ of a collection $(\phi^{\mathbb{P}})_{\mathbb{P} \in H}$ of tests for $\mathbb{P} \in H$.
            \added{Here, $\essinf_{\mathbb{P} \in H}$ denotes the greatest \emph{measurable lower bound} to ensure the measurability of the infimum (see Remark \ref{rmk:essinf} for more details).}
            While we believe this proposition is known, we are only aware of a sequential analogue that appears in \citet{ramdas2022admissible}. 

            \begin{prp}\label{prp:e_to_essinf_e}
                A test $\phi^H$ is valid for $H$ if and only if $\phi^H = \essinf_{\mathbb{P} \in H} \phi^{\mathbb{P}}$, for some collection $(\phi^{\mathbb{P}})_{\mathbb{P} \in H}$ with elements $\phi^{\mathbb{P}}$ that are individually valid for $\mathbb{P}$.
            \end{prp}

            To obtain the composite generalization of Theorem \ref{thm:simple}, we now induce a sequential test $(\phi_n^{\mathbb{P}})$ from each test $\phi^{\mathbb{P}}$, and construct the sequential test $(\phi_n^H)$ for $H$ as the running infimum:
            \begin{align*}
                \phi_n^H := \essinf_{\mathbb{P} \in H} \phi_n^{\mathbb{P}}
                    \equiv \essinf_{\mathbb{P} \in H} \mathbb{E}^{\mathbb{P}}[\phi^{\mathbb{P}} \mid \mathcal{F}_n].
            \end{align*}
            Moreover, to obtain a counterpart to the second claim of Theorem \ref{thm:simple} we must define the limiting test as $\phi_\infty^H := \essinf_{\mathbb{P} \in H} \phi_\infty^{\mathbb{P}}$ here.\footnote{We generally only have $\phi_\infty^H \geq \limsup_{n \to \infty} \phi_n^H$ because the limsup and essinf need not commute.}

            \begin{thm}[Composite hypotheses]\label{thm:composite}
                For every $\mathbb{P} \in H$, let $\phi^{\mathbb{P}}$ be an $\mathcal{F}$-measurable test.
                \begin{enumerate}
                    \item If every test $\phi^{\mathbb{P}}$ is valid, then $(\phi_n^H)$ is anytime valid.
                    \item \added{Suppose every $\phi^{\mathbb{P}}$ is $\mathcal{F}_\infty$-measurable. If the test $\phi^H \equiv \essinf_{\mathbb{P} \in H} \phi^{\mathbb{P}}$ is not valid, then $(\phi_n^H)$ is not anytime valid.}
                    \item If every test $\phi^{\mathbb{P}}$ is \added{$\mathcal{F}_\tau$-measurable}, then 
                    \added{$\phi_\tau^H = \essinf_{\mathbb{P} \in H} \phi^{\mathbb{P}}$}.
                \end{enumerate}
            \end{thm}

        \added{
        \begin{rmk}[Extending to e-values]\label{rmk:extending_to_e-values}
            We choose to present these results in terms of tests, as these are the most well-known kind of evidence variable \citep{koning2024posthocalphahypothesistesting}.
            However, we stress that most of the results in this section easily generalize to e-values through the bridge between e-values recently established in \citet{koning2024continuoustesting}, which we describe in Section \ref{sec:e-values}.
            Two exceptions are the second claim in Theorem \ref{thm:simple} and Proposition \ref{prp:admissible}, which rely on the uniform integrability that is automatically guaranteed by the uniform boundedness of sequential tests.
        \end{rmk}
        }
        
        \begin{rmk}[Selecting the collection $(\phi^{\mathbb{P}})_{\mathbb{P} \in H}$]\label{rmk:pointwise_to_composite}
            Unfortunately, it is not always clear how to appropriately select the collection of tests $(\phi^{\mathbb{P}})_{\mathbb{P} \in H}$, or whether there even exists a unique `natural' choice.
            For example, we would like to apply our machinery to derive optimal sequential tests by applying it to the most powerful test $\phi^*$ for $H$ against some alternative $\mathbb{Q}$, or a log-utility-optimal \citep{larsson2024numeraire} or more generally utility-optimal test \citep{koning2024continuoustesting}.
            Copying the approach for the $z$-test, we could naively select each $\phi^{\mathbb{P}}$ as the optimal test between $\mathbb{P}$ and $\mathbb{Q}$, which may be viewed as a generalization of the `universal inference' approach of \citet{wasserman2020universal} beyond log-utility-optimal e-values.
            However, this does not guarantee that $\essinf_{\mathbb{P} \in H} \phi^\mathbb{P} = \phi^*$.
            That is, the optimal test for $H$ against $\mathbb{Q}$ is not necessarily the essential infimum of the optimal tests for $\mathbb{P} \in H$ against $\mathbb{Q}$.

            A naive approach that does ensure $\essinf_{\mathbb{P} \in H} \phi^\mathbb{P} = \phi^*$ is to select $\phi^\mathbb{P} = \phi^*$ for every $\mathbb{P}$.
            Unfortunately, plugging this into Theorem \ref{thm:composite} often results in an uninteresting sequential test of the form $0, 0, \dots, 0, \phi^*$ or $\alpha, \alpha, \dots, \alpha, \phi^*$.
        \end{rmk}

        \added{
        \begin{rmk}[Existence of (essential) infimum]\label{rmk:essinf}
            A technical problem arises if the hypothesis $H$ is uncountable.
            In that case, we would like to take the infimum over an uncountable collection of $\mathcal{F}_n$-measurable functions, which is itself possibly not $\mathcal{F}_n$-measurable stemming from the fact that a $\sigma$-algebra is only guaranteed to be closed under countable unions.
            Assuming the $\sigma$-algebra is also closed under uncountable unions is one approach, but this may ban important applications.
            The route we take to avoid this problem is to replace the infimum by the essential infimum $\essinf_\mathbb{P} Y^{\mathbb{P}}$: the greatest \emph{measurable} almost-sure lower bound of the collection $(Y^{\mathbb{P}})$.
            Such an essential infimum exists if we assume $H$ is locally absolutely continuous with respect to some reference measure (see e.g. Appendix A.1 in \citet{ramdas2022admissible}).
            Another approach is to avoid taking an infimum by defining $(\phi_n^H)$ as some process such that $\phi_n^H \leq \mathbb{E}^{\mathbb{P}}[\phi^{\mathbb{P}} \mid \mathcal{F}_n]$ for every $\mathbb{P} \in H$.
            Unfortunately, this makes it difficult to derive a satisfying version of the third claim in Theorem \ref{thm:composite}.
            We acknowledge that this is somewhat unsatisfying, but we are not sure whether a satisfying solution exists within standard analysis.
        \end{rmk}
        }

    \subsection{Sufficient Monotone Likelihood Ratio}\label{sec:mlr}
        As highlighted in Remark \ref{rmk:pointwise_to_composite}, it is not always clear how to pick the collection of tests $(\phi^{\mathbb{P}})_{\mathbb{P} \in H}$.
        For example, for the one-sided $t$-test for the composite hypothesis that the effect size $\delta$ is smaller than 0, we may be tempted to choose each $\phi^\delta$ as the one-sided $t$-test for the simple hypothesis that the effect size equals $\delta$.
        While applying the infimum-construction in Theorem \ref{thm:composite} to this collection of tests will lead to a sequential test that is valid for the composite hypothesis and ends at the $t$-test, it is not admissible.

        In this section, we discuss an alternative approach to tackle composite hypotheses recently laid out by \citet{grunwald2025supermartingales}, in the presence of a Monotone Likelihood Ratio with respect to a sufficient statistic.
        They use this to derive a log-utility-optimal sequential version of the $t$-test \added{(see Section \ref{sec:e-values} for a primer on utility-optimal testing)}.
        We present a generalization of their result in Theorem \ref{thm:mlr} that allows us to derive a sequential version of the traditional $t$-test and more general utility-optimal variants of the $t$-test in Section \ref{sec:t-test}.

        In the result, we consider a model $(\mathbb{P}_\delta)_{\delta \in \Delta}$, where $\Delta$ is a totally ordered set, which we assume admits some dominating measure.

        \begin{thm}\label{thm:mlr}
            Assume for every $n$, when the probabilities are restricted to $\mathcal{F}_n$,
            \begin{enumerate}
                \item the Monotone Likelihood Ratio property holds in $\Delta$, with respect to a real-valued sufficient statistic $T_n$,
                \item the test $\phi_n$ is non-decreasing in $T_n$.
            \end{enumerate}
            If $(\phi_n)$ is a non-negative martingale for $\mathbb{P}_{\delta^0}$ starting at $\alpha$, then it is anytime valid for $(\mathbb{P}_{\delta^-})_{\delta^- \leq \delta^0}$.
        \end{thm}

        Compared to \citet{grunwald2025supermartingales}, the generalization is that we merely assume the tests $\phi_n$ are non-decreasing in the sufficient statistic $T_n$, whereas they assume they are likelihood ratios, and so non-decreasing in $T_n$ by the Monotone Likelihood Ratio property.
        We omit the proof, as it is easily obtained from the proof of \citet{grunwald2025supermartingales} by noticing that the non-decreasing-in-$T_n$ property suffices in the proof of their Proposition 2.

        While the second condition of Theorem \ref{thm:mlr} is quite straightforward to handle when building up a martingale forwards in time, it is not always easy to verify when inducing a sequential test with our Doob martingale construction, as in Theorem \ref{thm:simple}.
        For this reason, we provide sufficient conditions in Proposition \ref{prp:condition_2}, which are easier to verify.

        \begin{prp}\label{prp:condition_2}
            Let $T$ and $T_n$ be real-valued statistics.
            Assume $\mathcal{F} := \sigma(T)$ and $\mathcal{F}_n := \sigma(T_n)$, $\mathcal{F}_n \subseteq \mathcal{F}$, for every $n \geq 0$.
            Let $\phi$ denote an $\mathcal{F}$-measurable test.
            Assume that under $\mathbb{P}_{\delta^0}$,
            \begin{enumerate}
                \item[1.] $\phi$ is non-decreasing in $T$,
                \item[2.] $T$ is stochastically non-decreasing in $T_n$,
            \end{enumerate}
            Then, $\phi_n = \mathbb{E}^{\mathbb{P}_{\delta^0}}[\phi \mid \mathcal{F}_n]$ is non-decreasing in $T_n$.
        \end{prp}

        The first condition in Proposition \ref{prp:condition_2} is very natural.
        For example, it holds if the Monotone Likelihood Ratio property holds with respect to $T$, and $\phi$ maximizes \emph{some} expected utility $U$ under \emph{some} alternative $\mathbb{P}_{\delta^+}$, $\delta^+ \geq \delta^0$,
        \begin{align*}
            \max_{\phi : \textnormal{valid}} \mathbb{E}^{\mathbb{P}_{\delta^+}}[U(\phi)],
        \end{align*}
        where $U$ is assumed to be some concave and non-decreasing utility function that satisfies some regularity conditions.
        \added{Indeed, as we briefly discuss in Section \ref{sec:e-values}}, \citet{koning2024continuoustesting} shows that under such conditions the optimizing test $\phi$ is non-decreasing in the likelihood ratio between $\mathbb{P}_{\delta^0}$ and $\mathbb{P}_{\delta^+}$, which is non-decreasing in $T$ by the Monotone Likelihood Ratio property so that $\phi$ is also non-decreasing in $T$.

        Optimizing some utility against some alternative is quite a weak condition that seems reasonable in most applications of Theorem \ref{thm:mlr}.
        This means only the second condition truly needs to be checked, for which we can use Proposition \ref{prp:condition_2}.
        One may suspect that the second condition in Proposition \ref{prp:condition_2} is somehow implied by the assumptions of Theorem \ref{thm:mlr} if we additionally assume that a Monotone Likelihood Ratio property holds with respect to a sufficient statistic $T$.
        \added{Unfortunately this does not guarantee that $T$ is stochastically non-decreasing in $T_n$, which we demonstrate with a counterexample in Appendix \ref{app:counter}}.

      \section{Application to the $z$-test} \label{sec:ztest}
        In this section, we illustrate our methods in the context of a one-sided $z$-test.
        Suppose that $X_1, \dots, X_N$ are independently drawn from the normal distribution $\mathcal{N}(\mu,\sigma^2)$, with mean $\mu \in \mathbb{R}$ and $\sigma>0$.
        Then, the uniformly most powerful test for testing the null hypothesis $\mu = 0$ against the alternative hypothesis $\mu > 0$ is the one-sided $z$-test.

        At a given level $\alpha > 0$, this test equals
        \begin{align*}
            \phi_N 
                = \mathbb{I}\left\{\frac{1}{\sqrt{N}} \sum_{i=1}^N \frac{X_i}{\sigma} > z_{1-\alpha}\right\},
        \end{align*}
        where $z_{1-\alpha}$ is the $\alpha$ upper-quantile of the distribution $\mathbb{P} = \mathcal{N}(0, 1)$.
        By Theorem \ref{thm:simple}, for $n < N$, the induced anytime valid sequential test is given by
        \begin{align*}
            \phi_n'
                &=  \mathbb{P}\left(\frac{1}{\sqrt{N}}\sum_{i=1}^N \frac{X_i}{\sigma} > z_{1-\alpha} \mid X_1=x_1, \dots, X_n=x_n\right) \\
                &= \mathbb{P}\left(\frac{1}{\sqrt{N}}\sum_{i=n+1}^N \frac{X_i}{\sigma} > z_{1-\alpha} - \tfrac{1}{\sqrt{N}}\sum_{i=1}^n \frac{x_i}{\sigma}\right)
                = \Phi\left(\frac{\sum_{i=1}^n \frac{x_i}{\sigma} -  \sqrt{N}z_{1-\alpha}}{\sqrt{N - n}}\right),
        \end{align*}
        where $\Phi$ is the CDF of the standard normal distribution $\mathcal{N}(0, 1)$, and $x_1, \dots, x_n$ are the realizations of $X_1, \dots, X_n$.
        If $n = N$, this is poorly defined as the denominator equals zero, but still works if we define $y / 0$ as $+\infty$ if $y > 0$ and as $-\infty$ if $y < 0$.

        To interpret the sequential test $\phi_n'$ and compare it to the $z$-test at time $N$, $\phi_N$, it helps to write it as
        \begin{align}\label{eq:seq_z-test_interpretable}
            \Phi\left(\frac{N^{-1/2} \sum_{i=1}^n \frac{x_i}{\sigma} - z_{1 - \alpha}}{\sqrt{r_n}}\right),
        \end{align}
        where $r_n := (N - n)/N$ is the proportion of remaining observations.
        
        Here, we see that the argument of $\Phi$ is simply the difference between the current progress on the $z$-score test statistic $N^{-1/2} \sum_{i=1}^n \frac{x_i}{\sigma}$ and the critical value $z_{1-\alpha}$, divided by the square-root of the proportion $r_n$ of observations that remain.
        This means that if there are few observations remaining, so $r_n$ is close to 0, then the distance between the progress on the test statistic and critical value is inflated.
        If no observations remain, $r_n = 0$, the progress is inflated to either $+\infty$ or $-\infty$, depending on whether the test statistic exceeds the critical value or not.
        As $\Phi(+\infty) = 1$ and $\Phi(-\infty) = 0$, the function $\Phi$ then translates this to a rejection or non-rejection of the hypothesis.

        \subsection{The one-sided hypothesis $\mu \leq 0$}\label{sec:one-sided_z-test}
            So far, we have only considered the sequential $z$-test that is anytime valid for $\mu = 0$.
            We now explain why this same sequential test is anytime valid for the entire composite hypothesis $\mu \leq 0$.
    
            The idea is to construct the one-sided sequential $z$-test for every $\mu \leq 0$, and then plug these into Theorem \ref{thm:composite} to obtain an anytime valid test for the composite hypothesis that $\mu \leq 0$.
            We find that the resulting test coincides with the test for $\mu = 0$.
            
            The one-sided $z$-test for $\mu$ at level $\alpha > 0$ on the $[0, 1]$-scale equals
            \begin{align*}
                \phi^{\mu}
                    = \mathbb{I}\left\{\frac{1}{\sqrt{N}} \sum_{i = 1}^N \frac{X_i}{\sigma} > z_{1-\alpha} + \sqrt{N} \frac{\mu}{\sigma}\right\}.
            \end{align*}
            As $\mu \leq 0$, we have $\phi^0 = \inf_{\mu \leq 0} \phi^\mu$.
            Next, using $\mathbb{E}^{\mu}$ and $\mathbb{P}^\mu$ to denote the expectation and probability under $\mathcal{N}(\mu, \sigma^2)$, we have for $n < N$,
            \begin{align*}
                \phi_n^{\mu}
                    &= \mathbb{E}^{\mu}[\phi^\mu \mid X_1 = x_1, \dots, X_n = x_n]
                    = \mathbb{P}^{\mu}\left(\sum_{i = n + 1}^N \frac{X_i}{\sigma} > N^{1/2}z_{1-\alpha} +N\frac{\mu}{\sigma} - \sum_{i=1}^n \frac{x_i}{\sigma}\right) \\
                    &= \mathbb{P}^0\left(\sum_{i = n + 1}^N \frac{X_i + \mu}{\sigma} > N^{1/2}z_{1-\alpha} + N\frac{\mu}{\sigma} - \sum_{i=1}^n \frac{x_i}{\sigma} \right) \\
                    &= \mathbb{P}^0\left(\sum_{i = n + 1}^N \frac{X_i}{\sigma} > N^{1/2}z_{1-\alpha} +n\frac{\mu}{\sigma} - \sum_{i=1}^n \frac{x_i}{\sigma} \right)
                    = \Phi\left(\frac{N^{-1/2} \sum_{i=1}^n \frac{x_i}{\sigma} - z_{1-\alpha} - nN^{-1/2}\frac{\mu}{\sigma}}{\sqrt{\frac{N-n}{N}}}\right),
            \end{align*}
            which is the same as \eqref{eq:seq_z-test_interpretable}, but with an extra term involving $nN^{-1/2}\frac{\mu}{\sigma}$ in the numerator.
            The argument of $\Phi$ is decreasing in $\mu$, and $\Phi$ itself is an increasing function.
            As $\mu \leq 0$, this implies $\inf_{\mu \leq 0} \phi_n^\mu = \phi_n^0$; the sequential $z$-test for $\mu = 0$.
            Moreover, as each $\phi^\mu$ is valid for $\mathcal{N}(\mu, \sigma^2)$, the sequential test is valid for the composite hypothesis $\mu \leq 0$ by Theorem \ref{thm:composite}.

    \section{Application to the $t$-test}\label{sec:t-test}
        In this section, we apply the tools developed in Section \ref{sec:inducing_sequential_test} and Section \ref{sec:mlr} in particular, to construct an anytime valid sequential version of the one-sided $t$-test that is valid for the composite null hypothesis
        \begin{align*}
            H
                = \{\mathbb{P}_{\mu,\sigma}=\mathcal{N}(\mu,\sigma^2):\mu\leq0,\sigma^2>0\}.
        \end{align*}

        We also discuss how it can be easily computed, compare it to the sequentialized $z$-test, explain how this connects to the log-utility-optimal $t$-test studied by \citet{perez2024e-statistics}, \citet{wang2024anytime}, and \citet{grunwald2024safe}, and note how it can be applied beyond Gaussian distributions to general spherical distributions.
        The sequentialized traditional $t$-test is also studied by \citet{posch2004conditional}, but they only show its validity for $\mu = 0$.
        We also derive a much simpler expression by passing to the beta distribution.
        
        \subsection{Inducing a sequential test from the $t$-test}\label{sec:trej}
            Suppose we sequentially observe real-valued data $X_1, X_2, \dots$, which we stack into the tuples $X^n \eqdef (X_1, \dots, X_n)$, $n \geq 1$.
            Based on this data, we compute a sequence of $t$-statistics, $(T_n)_{n\geq1}$, where
            \begin{align*}
                T_n 
                = \frac{\frac{1}{\sqrt{n}}\sum_{i=1}^{n}X_i}{\sqrt{\frac{1}{n-1}\sum_{i=1}^{n}\left(X_i - \frac{1}{n}\sum_{j=1}^{n} X_j\right)^2}},
            \end{align*}
            for  $n\geq2$, and $T_1 = -\infty$ if $X_1 < 0$, $T_1 = 0$ if $X_1 = 0$ and $T_1 = +\infty$ if $X_1 > 0$.
            We consider the filtration $(\mathcal{F}_n)_{n \geq 0}$ induced by the $t$-statistics, $\mathcal{F}_n = \sigma\left(T_1,\dots,T_n\right)$.
            \added{Although the observations $X_1, X_2, \dots$ are i.i.d., the induced $t$-statistics are dependent and not identically distributed.}

            For $n > 1$ observations, the traditional one-sided $t$-test for hypothesis $H$ equals
            \begin{align*}
                \phi_N
                    = \mathbb{I}\left\{T_N >c_{\alpha,N}\right\},
            \end{align*}
            where $c_{\alpha,N}$ is the $\alpha$ upper-quantile of the $t$-distribution with $N-1$ degrees of freedom.
            
            Given some planned number of observations $N$, we can induce a sequential test
            \begin{align*}
               \phi_n' = \mathbb{E}^{\mathcal{N}(0,1)}\left[\phi_N \mid \mathcal{F}_n\right],
            \end{align*}
            for every $n \geq 0$.
            As $\phi_N$ only depends on $T_N$ and $T_N$ is known at time $N$ by construction, we have that at time $N$ this indeed coincides with the traditional $t$-test: $\phi_N' = \phi_N$.
            From Theorem \ref{thm:mlr} it immediately follows that it is anytime valid for $\mu = 0$ and $\sigma^2 = 1$, and since the statistics $T_1, \dots, T_N$ are scale-invariant it is anytime valid for the composite hypothesis $\mu = 0$ and $\sigma^2 > 0$.

            In Proposition \ref{prp:seqt}, we show that this same sequential test is in fact also valid for the composite hypothesis $H$, that $\mu \leq 0$ and $\sigma^2 > 0$.
            
            \begin{prp}\label{prp:seqt}
                The induced sequential t-test $(\phi_n')$ is anytime valid for $H$.
            \end{prp}

            We also highlight Lemma \ref{lem:increasing}, which is a key step in the proof of Proposition \ref{prp:seqt} that allows us to apply Proposition \ref{prp:condition_2}.
            Appendix \ref{app:proofseqt}, we even derive a simple expression for the full conditional distribution of $T_N$ given $T_n$.
            \begin{lem}\label{lem:increasing}
                Under $\mathcal{N}(0,\sigma^2),$ for any $\sigma^2>0$, $T_N$ is stochastically increasing in $T_n$, for every $N>1$, $n\leq N$.
            \end{lem}
            
            \begin{rmk}[Beyond the traditional $t$-test]
                An inspection of the proof of Proposition \ref{prp:seqt} shows that it only uses the fact that the traditional $t$-test is non-decreasing in the $t$-statistic $T_N$.
                This means that the result also applies to any other test that is non-decreasing in the t-statistic.
                This includes utility-optimal versions of the $t$-test, and the log-utility-optimal $t$-test derived by \citet{perez2024e-statistics}, and recently shown to be valid for the one-sided hypothesis by \citet{grunwald2025supermartingales}.
            \end{rmk}
            
            \begin{rmk}[Valid for spherical distributions]
                The tests here are actually not just valid for $H$, but even valid for a location-shift of a spherically invariant distribution.
                This is not surprising, as the t-test is valid under sphericity \citep{efron1969student}, and can even be viewed as a test for spherical invariance \citep{lehmann1949theory, koning2024more}.
                The link to the Gaussian distribution is that the i.i.d. multivariate Gaussian distribution is the only spherically invariant distribution with i.i.d. marginals, by the Herschel-Maxwell Theorem.
            \end{rmk}

        \subsection{Computation and comparison to $z$-test}
            A convenient expression for $\phi_n'$, which makes it straightforward to compute, is given by
            \begin{align*}
                \phi_n'
                    = \int_{0}^1 F_{B}\left(\frac{N^{-1/2} \sum_{i=1}^n \frac{x_i}{\|x^n\|_2} \sqrt{w} - \beta_{1-\alpha}}{\sqrt{r_n} \sqrt{1 - w}}\right) f(w)\, dw,
            \end{align*}
            where $F_B$ is the CDF of a Beta$\left(\frac{N-n-1}{2}, \frac{N-n-1}{2}\right)$-distribution on $[-1, 1]$, $\beta_{1-\alpha}$ is the $\alpha$ upper-quantile of a Beta$\left(\frac{N-1}{2},\frac{N-1}{2}\right)$-distribution on $[-1,1]$,  $r_n = (N-n) / N$ denotes the proportion of remaining observations, and $f$ is the density of a Beta$\left(\frac{n}{2}, \frac{N-n}{2}\right)$-distribution on $[0, 1]$.
            A derivation is given in Appendix \ref{app:deriv}.

            We have purposely written this expression in a way that makes it easy to compare to the expression we derived for the $z$-test:
            \begin{align*}
                \Phi\left(\frac{N^{-1/2} \sum_{i=1}^n \frac{x_i}{\sigma} - z_{1 - \alpha}}{\sqrt{r_n}}\right).
            \end{align*}
            There are two key differences.
            The first is that the $t$-test relies on the normalized data $x_i / \|x^n\|_2$ instead of the original data $x_i$, which in turn forces a switch to the Beta CDF and quantile.
            The second difference is the presence of $w$, which is due to the fact that, at time $n$, we do not know the relative length of the observed vector $x^n$ to the final vector $X^N$.
            For small $n$, the distribution of $w$ is more concentrated near 0, so that the $\sqrt{w}$-term `shrinks' the test statistic when $n$ is small.
            For large $n$, the distribution of $w$ is more concentrated near 1, so that the $\sqrt{w}$-term no longer shrinks the test statistic.
            At the same time, for large $n$, the $\sqrt{1 - w}$-term in the denominator inflates the distance between the test statistic and critical value when $n$ is large.

    \section{Test outcomes as conditional significance levels}\label{sec:test_as_significance_level}
        In this section, we formally discuss the idea of optional continuation: that one may use a $[0,1]$-valued test outcome as a significance level in subsequent testing.
        Our main claim is Theorem \ref{thm:seq_significance_level}, but we first present a simpler version with only two tests in Theorem \ref{thm:seq_significance_level_two}.

        Theorem \ref{thm:seq_significance_level_two} shows that a seemingly `inconclusive' test outcome $\phi_1 \in (0, 1)$ may be taken by any other person and used as a starting significance level in a subsequent experiment.
        
        \begin{thm}[Two tests]\label{thm:seq_significance_level_two}
            Suppose that $\phi_1$ is valid at level $\alpha$ and $\phi_2$ is conditionally valid at conditional level $\phi_1$: $\mathbb{E}^{\mathbb{P}}[\phi_2 \mid \phi_1] \leq \phi_1$, for every $\mathbb{P} \in H$.
            Then, $\phi_2$ is valid at level $\alpha$.
        \end{thm}
        \begin{proof}
            For every $\mathbb{P} \in H$, we have $\mathbb{E}^{\mathbb{P}}[\phi_2]  = \mathbb{E}^{\mathbb{P}}[\mathbb{E}^{\mathbb{P}}[\phi_2\mid \phi_1]] \leq \mathbb{E}^{\mathbb{P}}[\phi_1] \leq \alpha$.
        \end{proof}

        Theorem \ref{thm:seq_significance_level} generalizes this to an anytime validity context.
        Here, we show that we may abort an anytime valid sequential test at an arbitrary stopping time $\tau$, and then continue experimentation given the present information by starting up a new sequential test that is conditionally anytime valid at the conditional significance level $\phi_\tau$.
        This may be used to switch to a different sequential test, used to redesign the experiment, or even used as a starting point in (unplanned) future experiments (see Remark \ref{rmk:branching}).
        
        \begin{thm}[Anytime validity and optional continuation]\label{thm:seq_significance_level}
            Suppose that $(\phi_n)$ is an anytime valid sequential test at level $\alpha$, and $\sigma$ is an arbitrary stopping time, both adapted to the filtration $(\mathcal{F}_n)$.
            Let $\phi^*$ be an $\mathcal{F}$-measurable test that is valid at the \added{conditional} significance level $\phi_{\sigma}$ conditional on $\mathcal{F}_{\sigma}$, 
            \begin{align*}
             \mathbb{E}^{\mathbb{P}}[\phi^* \mid \mathcal{F}_{\sigma}] \leq \phi_{\sigma}, \quad \added{\text{for every $\mathbb{P}\in H$}}.
            \end{align*}
            Then, $\phi^*$ is valid.
            
            \added{Suppose that the sequential test $(\phi_n^*)_{n \geq \sigma}$ adapted to $(\mathcal{F}_n)_{n \geq \sigma}$ is conditionally anytime valid at conditional significance level $\phi_\sigma$: for every stopping time $\tau$ with respect to $(\mathcal{F}_n)_{n \geq \sigma}$,
            \begin{align*}
                \mathbb{E}^{\mathbb{P}}\left[\phi^*_\tau \mid \mathcal{F}_\sigma\right] \leq \phi_\sigma, \quad \textnormal{ for every } \mathbb{P}\in H.
            \end{align*}
            Then, the sequential test $(\phi_n^\dagger) := \phi_1, \dots, \phi_\sigma, \phi_{\sigma + 1}^*, \phi_{\sigma + 2}^*, \dots$ is anytime valid.}      
        \end{thm}

        \added{
        \begin{rmk}[Increasing power with optional continuation?]\label{rmk:increasing_power}
            One may wonder whether optional continuation may be used to obtain a more powerful testing procedure.
            To answer this question, we must be careful about how we define `power': the sequential context forces us to be explicit about the time at which the power is attained.
    
            To highlight this, let us consider the most powerful test $\phi_N$ at time $N$ for some simple hypothesis $\mathbb{P}$.
            Now, its induced sequential test $(\phi_n')$ is implicitly performing optional continuation under the hood at every time $n$: $\mathbb{E}^{\mathbb{P}}[\phi_n' \mid \mathcal{F}_{n-1}] = \phi_{n-1}'$, since $(\phi_n')$ is a martingale.
            As $\phi_N$ is assumed to be most powerful, straying from this implicit optional continuation cannot increase power at time $N$.
            However, with optional continuation we can obtain more power at an earlier or later (stopping) time $\tau \neq N$.
        \end{rmk}
        }
        
        \added{
        \begin{rmk}[Branching filtrations]\label{rmk:branching}
            When we redesign the experiment at a stopping time $\tau$, we implicitly branch from the planned filtration $(\mathcal{F}_n)$ into some new filtration $(\mathcal{F}_n')$ which overlaps with $(\mathcal{F}_n)$ at every time $n \leq \tau$, where $\mathcal{F}_n, \mathcal{F}_n' \subseteq \mathcal{F}$ for all $n$.
            Now, one may worry that we would need to generalize sequential tests and anytime validity to such `data-dependent branching filtrations'.
            
            Luckily, we may view this branching process as implicitly defining a bona fide filtration.
            To understand this, we may capture the decision to branch by some stopping time $\tau$ and an $\mathcal{F}_\tau$-measurable decision variable $D \in \{0, 1\}$, where $D = 0$ indicates we remain in our original filtration $(\mathcal{F}_n)$ and $D = 1$ that we branch into a new filtration $(\mathcal{F}_n')$.
            This then implicitly defines the filtration $(\mathcal{G}_n)$ as
            \begin{align*}
                \mathcal{G}_n
                    = \left\{F \in \mathcal{F} : F \cap \{n \leq \tau\} \in \mathcal{F}_n, F \cap \{n > \tau, D = 0\} \in \mathcal{F}_n, F \cap \{n > \tau, D = 1\} \in \mathcal{F}_n'\right\},
            \end{align*}
            for every $n$, so that we can view our sequential tests as adapted to the bona fide filtration $(\mathcal{G}_n)$.
            This idea may be iteratively applied to generalize it to branching decisions at multiple stopping times, and is easily expanded to more than two branching options.
        \end{rmk}
        }

        \begin{rmk}[Binary tests]
            Binary \added{($\{0, 1\}$-valued)} tests, which are common in practice, are usually not the ideal choice for the first test $\phi_1$ in Theorem \ref{thm:seq_significance_level_two}.
            \added{Indeed, if $\phi_1 = 0$, then initializing the second test at conditional} significance level 0 will never lead to a subsequent rejection.
            Moreover, if $\phi_1 = 1$, then we have already attained the desired rejection at level $\alpha$.
            We suspect this may be the reason why this idea has not been picked up before, as binary tests are highly popular in practice.
        \end{rmk}
        
        \begin{rmk}[Relationship to multiplying e-values]\label{rmk:multiply_e-values}
            We may perform this same exercise with e-values; see Section \ref{sec:e-values} for a primer on e-values.
            There, we find this is equivalent to multiplying sequentially valid e-values.
            
            In the context of Theorem \ref{thm:seq_significance_level_two}, the trick is to shift to the evidence scale by making the substitution $\varepsilon_1 = \phi_1/\alpha$ and $\varepsilon_2 = \phi_2/\phi_1$.
            After observing the outcome of the level $\alpha$ test $\varepsilon_1$, we choose a valid level $\alpha\varepsilon_1$ test $\varepsilon_2$ taking value in $[0, 1/(\alpha \varepsilon_1)]$,
            \begin{align*}
                \mathbb{E}^{\mathbb{P}}\left[\varepsilon_2 \mid \varepsilon_1 \right] \leq 1, \quad \textnormal{for every } \mathbb{P} \in H.
            \end{align*}
            The combined evidence then equals the product $\varepsilon_1\varepsilon_2$, which is indeed valid:
            \begin{align*}
                \sup_{\mathbb{P} \in H} \mathbb{E}^{\mathbb{P}}\left[\varepsilon_1\varepsilon_2\right]
                    &= \sup_{\mathbb{P} \in H} \mathbb{E}^{\mathbb{P}}\left[\varepsilon_1 \mathbb{E}^{\mathbb{P}}[\varepsilon_2 \mid \varepsilon_1]\right]
                    \leq \sup_{\mathbb{P} \in H} \mathbb{E}^{\mathbb{P}}\left[\varepsilon_1\right]
                    \leq 1.
            \end{align*}
            Moreover, it is of level $\alpha$ as $\varepsilon_1\varepsilon_2$ is $[0, \varepsilon_1 \times 1/(\varepsilon_1\alpha)] = [0, 1/\alpha]$-valued.
        \end{rmk}

        \subsection{Switching significance levels}\label{sec:switching_alpha}
            A natural question to ask is whether we need to stick to the significance level $\alpha$ that we set out with at the start, or whether we may adaptively replace this with some other significance level $\alpha'$.
            It turns out that we may indeed adaptively change it, as long as we are careful with formulating the resulting unconditional Type I error guarantee.
            
            For simplicity, we stick to the setting of Theorem \ref{thm:seq_significance_level_two}, where we first perform test $\phi_1$ at level $\alpha$, and want to subsequently conduct a test $\phi_2$ so that the unconditional significance level of $\phi_2$ becomes $\alpha'$.

            Let us start with the simple case, where the new significance level $\alpha'$ is chosen independently from the data.
            Then, we can simply choose $\phi_2$ to have conditional level $\phi_1 \alpha'/\alpha$,
            \begin{align*}
                \mathbb{E}^{\mathbb{P}}[\phi_2 \mid \mathcal{F}_1] \leq \phi_1 \frac{\alpha'}{\alpha}, \textnormal{ for every } \mathbb{P} \in H,
            \end{align*}
            and follow the proof of Theorem \ref{thm:seq_significance_level_two} to establish for every $\mathbb{P} \in H$ that
            \begin{align*}
                \mathbb{E}^{\mathbb{P}}[\phi_2]
                    &= \mathbb{E}^{\mathbb{P}}[\mathbb{E}^{\mathbb{P}}[\phi_2\mid \mathcal{F}_1]]
                    \leq \mathbb{E}^{\mathbb{P}}\left[\phi_1 \frac{\alpha'}{\alpha}\right]
                    \leq \frac{\alpha'}{\alpha} \mathbb{E}^{\mathbb{P}}[\phi_1] 
                    \leq  \alpha'.
            \end{align*}

            However, we would like to allow $\alpha'$ to depend on the data we have collected so far.
            Unfortunately, the desire to have the unconditional Type I error guarantee on $\phi_2$,
            \begin{align*}
                \mathbb{E}^{\mathbb{P}}[\phi_2] \leq \alpha', \textnormal{ for every } \mathbb{P} \in H,
            \end{align*}
            is then a strange requirement, as $ \mathbb{E}^{\mathbb{P}}[\phi_2]$ is $\mathcal{F}_0$-measurable and $\alpha'$ is $\mathcal{F}_1$-measurable.
            Luckily, this can be resolved by shifting to data-dependent Type I error control recently developed in \citet{koning2024posthocalphahypothesistesting} and \citet{grunwald2024beyond},
            \begin{align*}
                \mathbb{E}^{\mathbb{P}}\left[\frac{\phi_2}{\alpha'}\right] \leq 1, \textnormal{ for every } \mathbb{P} \in H.
            \end{align*}
            Under this generalized Type I error control, initializing $\phi_2$ to be valid at conditional level $\phi_1 \alpha'/\alpha$ indeed makes it unconditionally valid at level $\alpha'$:
            \begin{align*}
                \mathbb{E}^{\mathbb{P}}\left[\frac{\phi_2}{\alpha'}\right]
                    &= \mathbb{E}^{\mathbb{P}}\left[\mathbb{E}^{\mathbb{P}}\left[\frac{\phi_2}{\alpha'} \middle| \mathcal{F}_1 \right]\right]
                    \leq \mathbb{E}^{\mathbb{P}}\left[\frac{\phi_1}{\alpha}\right]
                    \leq \frac{1}{\alpha} \mathbb{E}^{\mathbb{P}}[\phi_1] 
                    \leq 1, \textnormal{ for every } \mathbb{P} \in H.
            \end{align*}

        \added{
        \subsection{Pointwise optional continuation}\label{sec:pointwise_continuation}
            In the preceding sections, we use the outcome of a single test $\phi_1$ for optional continuation.            
            Recall that from Proposition \ref{prp:e_to_essinf_e} we may represent a test $\phi_1^H$ for a composite hypothesis $H$ as the infimum $\essinf_{\mathbb{P} \in H} \phi_1^{\mathbb{P}}$ of a family of tests $(\phi_1^{\mathbb{P}})_{\mathbb{P} \in H}$.
            The insight in this section is that for the purpose of optional continuation it seems generally better to work with such a family $(\phi_1^{\mathbb{P}})_{\mathbb{P} \in H}$ than working directly with $\phi_1^H$.
            
            The idea is to use optional continuation for each test separately:
            \begin{align}\label{ineq:pointwise_continuation}
                \mathbb{E}^{\mathbb{P}}[\phi_2^{\mathbb{P}} \mid \mathcal{F}_1] \leq \phi_1^{\mathbb{P}}.
            \end{align}
            In Theorem \ref{thm:pointwise_continuation}, we show that this still results in a valid test $\phi_2^H := \essinf_{\mathbb{P} \in H} \phi_2^{\mathbb{P}}$.
            Moreover, this approach dominates using $\phi_1^H \equiv \essinf_{\mathbb{P} \in H} \phi_1^{\mathbb{P}}$ for $H$ as a conditional significance level for $\phi_2^H$, since \eqref{ineq:pointwise_continuation} implies $\essinf_{\mathbb{P} \in H} \mathbb{E}^{\mathbb{P}}[\phi_2^\mathbb{P} \mid \mathcal{F}_1] \leq \essinf_{\mathbb{P} \in H} \phi_1^\mathbb{P} \equiv \phi_1^{H}$, which in turn implies that $\phi_2^H$ is conditionally valid at conditional level $\phi_1^H$ since
            \begin{align*}
                \mathbb{E}^{\mathbb{P}}[\phi_2^H \mid \mathcal{F}_1]
                    \equiv \mathbb{E}^{\mathbb{P}}[\essinf_{\mathbb{P}' \in H} \phi_2^{\mathbb{P}'} \mid \mathcal{F}_1]
                    \leq \essinf_{\mathbb{P} \in H} \mathbb{E}^{\mathbb{P}}[\phi_2^\mathbb{P} \mid \mathcal{F}_1]
                    \leq \phi_1^{H}.
            \end{align*}
            But conditional validity of $\phi_2^H$ at conditional level $\phi_1^H$ clearly does not generally imply \eqref{ineq:pointwise_continuation}.

            A consequence is that for optional continuation one should perhaps report the entire family $(\phi_1^{\mathbb{P}})_{\mathbb{P} \in H}$ of test outcomes instead of just a single test outcome $\phi_1^H$ against $H$.
            
            \begin{thm}\label{thm:pointwise_continuation}
                Let $(\phi_1^{\mathbb{P}})_{\mathbb{P} \in H}$ be a family of tests, where each test $\phi_1^{\mathbb{P}}$ is valid for $\mathbb{P}$ at level $\alpha$.
                Moreover, let $(\phi_2^{\mathbb{P}})_{\mathbb{P} \in H}$ be a family of tests, where each test $\phi_2^{\mathbb{P}}$ is conditionally valid for $\mathbb{P}$ at conditional significance level $\phi_1^{\mathbb{P}}$, given $\mathcal{F}_1$.
                Then the test $\phi_2^H = \essinf_{\mathbb{P} \in H} \phi_2^{\mathbb{P}}$ is valid for $H$ at level $\alpha$.
            \end{thm}
            }
            
    \section{Connection to e-values}\label{sec:e-values}
        \added{
        \subsection{E-values and tests}
            Our favourite way to introduce the e-value is as a \emph{test on a different scale}.
            In particular, without loss of generality, we could choose to define a level $\alpha > 0$ test as a map $\varepsilon : \mathcal{X} \to [0, 1/\alpha]$, rescaling by a factor $1/\alpha$ compared to the classical definition $\phi : \mathcal{X} \to [0, 1]$.
            Because of the rescaling, the validity guarantee now becomes
            \begin{align*}
                \sup_{\mathbb{P} \in H} \mathbb{E}^{\mathbb{P}}\left[\varepsilon\right] \leq 1,
            \end{align*}
            regardless of the value of $\alpha$.
            A test on this different scale is known as a (bounded) \emph{e-value}.
            Plugging in $\alpha = 0$ recovers the usual definition of the e-value $\varepsilon : \mathcal{X} \to [0, \infty]$.
            From this perspective, recently proposed in \citet{koning2024continuoustesting}, e-values are just tests in disguise, masked by a cosmetic rescaling.
            As noted in Remark \ref{rmk:extending_to_e-values} and Remark \ref{rmk:multiply_e-values}, one consequence is that most of our results easily extend to e-values.

            As the lack of an upper bound is of little practical relevance, \citet{koning2024continuoustesting} argues that the true innovation of the e-value is an appreciation of the interior of the codomain $[0, 1]$ (now $[0, 1/\alpha]$).
            While this interior is classically viewed as a (much disliked) instruction to randomize, the e-value literature (implicitly) directly interprets this as `evidence'.
            Our work delivers a new interpretation, showing that test outcomes in $(0, 1)$ can be used as a conditional significance level.

            The appreciation of this interior is visible in the power objectives considered in the e-value literature, which will be important for the results in the subsequent sections.
            To understand this, let us start with classical Neyman-Pearson optimal testing, which concerns maximizing the power $ \mathbb{E}^{\mathbb{Q}}[\varepsilon]$ over $[0, 1/\alpha]$-valued valid e-values against some alternative $\mathbb{Q}$.
            An optimizer of this problem can also be viewed as optimizing
            \begin{align*}
                \mathbb{E}^{\mathbb{Q}}\left[U(\varepsilon)\right]
            \end{align*}
            over $[0, \infty]$-valued valid e-values, for the `Neyman-Pearson utility function' $U(x) = x \wedge 1/\alpha$ \citep{koning2024continuoustesting}.
            This problem is optimized by the likelihood ratio test, which is usually $\{0, 1/\alpha\}$-valued, offering either no rejection or a rejection at level $\alpha$.
            Other concave utility functions $U : [0, \infty] \to [-\infty, +\infty]$, designed to express the utility associated with different outcomes, generally yield e-values that take value in a continuum.

            One choice of the utility function is the log-utility $U = \log$, which has grown to become the (near-universal) default in the e-value literature \citep{grunwald2024safe, larsson2024numeraire} and will be the primary focus of the following sections.
            For a simple hypothesis, $H = \{\mathbb{P}\}$ and $\alpha = 0$, the log-utility-optimal e-value is simply the likelihood ratio $d\mathbb{Q} / d\mathbb{P}$.
            For mutually absolutely continuous $\mathbb{P}$ and $\mathbb{Q}$ and upper-semicontinuous, concave and non-decreasing utility functions $U$, \citet{koning2024continuoustesting} shows that an expected-utility-optimal e-value $\varepsilon^*$ satisfies
            \begin{align}\label{eq:U-optimal_e-value}
                \lambda \frac{d\mathbb{P}}{d\mathbb{Q}} \in \partial U(\varepsilon^*),
            \end{align}
            for some normalization constant $\lambda \geq 0$, where $\partial U$ denotes the super-differential of $U$.
            This means that utility-optimal e-values are generally non-decreasing as a function of the likelihood ratio $d\mathbb{Q}/d\mathbb{P}$.

        \subsection{Inducing from log-utility-optimal e-values}\label{sec:log_induction}
            In Theorem \ref{thm:log_invariance}, we apply our machinery from Section \ref{sec:inducing_sequential_test} to log-utility-optimal e-values.
            This reveals a remarkable invariance property of log-utility-optimality: if $\varepsilon$ is log-utility-optimal for $\mathbb{P}$ against $\mathbb{Q}$, then the $\mathcal{F}_n$-conditional e-value $\varepsilon_{n} = \mathbb{E}^{\mathbb{P}}[\varepsilon \mid \mathcal{F}_n]$ is automatically log-utility-optimal for the same problem restricted to the sub-$\sigma$-algebra $\mathcal{F}_n$.
            The resulting log-utility-optimal sequential test (`e-process') $(\varepsilon_n) = (\textnormal{LR}_n)$ can be viewed as the `unbounded SPRT' on the e-value scale for $\alpha = 0$: $\textnormal{LR}_n \wedge 1/0 = \textnormal{LR}_n$, taking $1/0 = +\infty$.
            
            In the result, we use $\mathbb{P}_{\mathcal{F}_n}$ and $\mathbb{Q}_{\mathcal{F}_n}$ to denote the probabilities $\mathbb{P}$ and $\mathbb{Q}$ restricted to the sub-$\sigma$-algebra $\mathcal{F}_n$.
            The proof relies on the (well-known) property that the $\mathcal{F}_n$-restricted likelihood ratio $\textnormal{LR}_n$ equals the conditional expectation of the likelihood ratio $\textnormal{LR}$.

            \begin{thm}[Conditioning-invariance log-utility-optimal e-values]\label{thm:log_invariance}
                Let $\mathbb{P} \gg \mathbb{Q}$.
                If $\varepsilon$ is log-utility-optimal for $\mathbb{P}$ against $\mathbb{Q}$ then $\varepsilon_n$ is log-utility-optimal for $\mathbb{P}_{\mathcal{F}_n}$ against $\mathbb{Q}_{\mathcal{F}_n}$.
            \end{thm}

            In Proposition \ref{prp:define_log}, we show that the log-utility function is the \emph{only} utility function that satisfies this property.
            The main idea behind the proof is to use the representation \eqref{eq:U-optimal_e-value} of \citet{koning2024continuoustesting} to show that the property requires Jensen's inequality to hold with equality for the map $g : x \mapsto (U')^{-1}(c/x)$, $c > 0$, so that $g$ must be affine on the support of the distribution.
            If the class of distributions for which we want the property to hold is sufficiently rich, this implies $g$ must be affine.
            This subsequently implies that the utility function is of the form $U(x) = a\log(x - b) + c$, for which the optimizer is a mixture of the log-utility-optimal e-value and 1.
            Applying $U'(0) = \infty$ forces $b = 0$.
            
            \begin{prp}\label{prp:define_log}
                Assume that $U$ is concave, strictly increasing and differentiable with continuous and strictly decreasing derivative $U'$ with $U'(0) = \lim_{x \to 0} U'(x) = \infty$ and $U'(\infty) = \lim_{x \to \infty} U'(x) = 0$.
                Suppose that for every pair $\mathbb{P} \gg \mathbb{Q}$ and underlying measure space, we have that if $\varepsilon$ is $U$-optimal for $\mathbb{P}$ against $\mathbb{Q}$ then $\varepsilon_n$ is $U$-optimal for $\mathbb{P}_{\mathcal{F}_n}$ against $\mathbb{Q}_{\mathcal{F}_n}$.
                Then $U = \log$ (up to affine transformations).
            \end{prp}
            }
        \subsection{Inducing the SPRT from a Neyman-Pearson test}\label{sec:log-opt_z-test}
            Theorem~\ref{thm:LR_sequential_z-test} shows that under an i.i.d. assumption and mild regularity conditions, the unbounded SPRT (log-optimal e-process / test-process) $(\textnormal{LR}_n)$ is the sequential test induced by the Neyman-Pearson optimal test based on $N$ observations for small $\alpha$ as $N \to \infty$, on the e-value scale $\varepsilon_{n, \alpha} = \phi_{n, \alpha} / \alpha$.

            An interpretation of this result is that using the unbounded SPRT means we are implicitly aiming for a terminal test with high power at large $N$ and small $\alpha$.
            The result is illustrated in the Gaussian location setting in Figure~\ref{fig:LR_z-test}, where we see the sequential tests indeed nearly overlap, especially for $n \ll N$, even though $N$ is not so large here.
            
            \begin{thm}\label{thm:LR_sequential_z-test}
                Consider two distributions $\mathbb{P}$ and $\mathbb{Q}$ that are mutually absolutely continuous, and assume $\log \frac{d\mathbb{Q}}{d\mathbb{P}}$ is atomless.
                We denote the level $\alpha_N$ NP-test based on $N$ i.i.d. observations with $\phi_{N,\alpha_N}$, and its conditional expectation given $n$ observations with $\phi'_{n,\alpha_N}$.
                Choose
                    $\alpha_N = \exp\{-N \times D(\mathbb{Q}\|\mathbb{P})\}$ where $D(\cdot\|\cdot)$ denotes the KL divergence.
                
                Assume that there exists an $\epsilon > 0$ such that the $(1+\epsilon)$-order Rényi divergence $D_{1+\epsilon}(\mathbb Q\|\mathbb P)$ is finite, and that $\log\LR_1$ is not lattice-valued under $\mathbb{P}$. Then,
                \begin{align*}
                    \lim_{N\to\infty} \varepsilon_{n, \alpha_N}' \equiv \lim_{N\to\infty}\frac{\phi'_{n,\alpha_N}}{\alpha_N}= \mathrm{LR}_n,\qquad \text{$\mathbb{P}$-a.s., for all } n\in\mathbb{N}.
                \end{align*}
    
            \end{thm}

            \begin{figure}
                \centering
                \includegraphics[width=9.5cm]{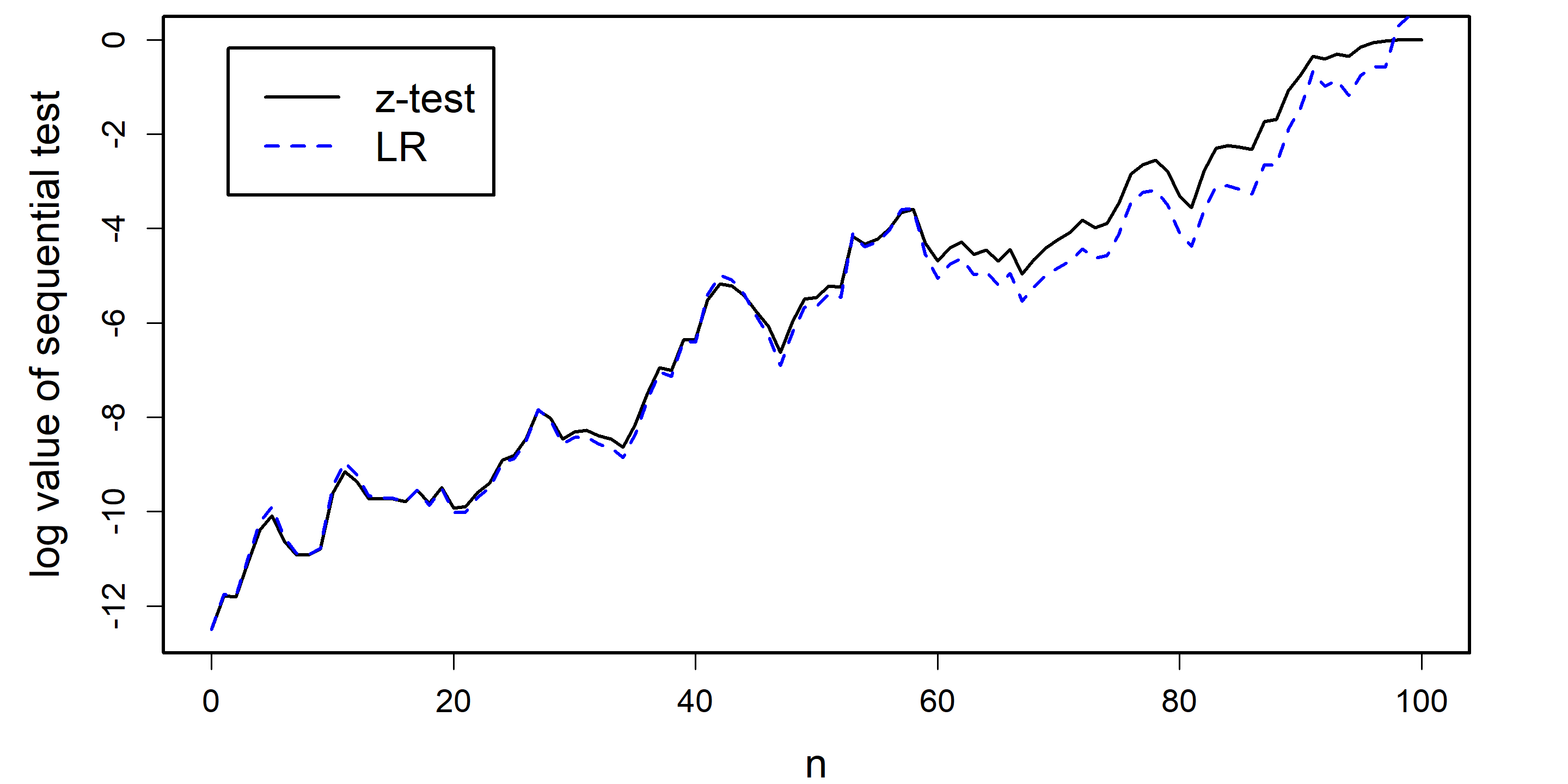}
                \caption{Illustration of the sequential $z$-test at $N = 100$ and the likelihood ratio process over 100 observations sampled from $\mathcal{N}(0.5, 1)$. 
                The $z$-test is executed at significance level $\alpha=\exp\{-N\frac{\mu^2}{2\sigma^2}\}=\exp\{-12.5\}\approx .0000037$, for which Theorem \ref{thm:LR_sequential_z-test} predicts that two sequential tests coincide as $N \to \infty$.}
                \label{fig:LR_z-test}
            \end{figure}
        
            \begin{rmk}[Link to \citet{breiman1961optimal}]
                Our Theorem~\ref{thm:LR_sequential_z-test} is related to a result by \cite{breiman1961optimal}.
                In particular, he shows in a binary-data context that the power of the SPRT $(\alpha \textnormal{LR}_n \wedge 1)_{n \in \mathbb{N}}$ at the final time $N$ asymptotically matches the power of the most powerful test at time $N$, uniformly in $\alpha$.
                Our result shows something much stronger: for a particular choice of $\alpha$ the entire test processes themselves coincide as $N \to \infty$, not just their power at time $N$.
            \end{rmk}
            
    \section{Discussion}
        \added{
        We want to start this discussion by stressing that we do not advocate for everyone to use sequential tests induced from classical Neyman-Pearson-style tests.
        Indeed, such tests satisfy two undesirable properties.
        First, they are measured at some finite time $N$, so that they do not naturally continue beyond this time.
        Second, such tests generally hit $0$ with positive probability at time $N$, discarding the option to continue beyond this time.
        This means we can only continue if the tests hits $1$, signifying a rejection at level $\alpha$, to hope for a rejection at a smaller level $\alpha' < \alpha$ as in Section \ref{sec:switching_alpha}.
        But if attaining a rejection at level $\alpha$ was our goal, then this may not be of interest.
        For these reasons, we primarily view the application of our work to such classical tests as establishing the theoretical bridge between classical testing and anytime valid sequential testing, which are often treated as different paradigms in the literature.
        That said, these sequential tests induced from Neyman-Pearson-style tests also expand classical testing by offering the option to intervene before time $N$ using optional continuation, allowing analysts to adapt their experiment based on the present situation.
        Given the wide use of classical tests and a general reluctance to move away from established methodology, we believe inducing sequential tests from Neyman-Pearson-style tests is a valuable addition to the statistical toolbox.
        }

        \added{
        So what do we advocate for?
        Our Proposition \ref{prp:admissible} shows that admissible anytime valid sequential tests $(\phi_n)$ can be fully summarized by their terminal test $\phi_\infty = \limsup_{n \to \infty} \phi_n$.
        This suggests that the development of a `good' sequential tests may perhaps be reduced to the development of a `good' terminal tests, and the subsequent induction of a sequential test from this terminal test.
        For this reason, we believe an interesting area of future research is the development of appropriate terminal tests, which for the reasons given above will likely not be of the Neyman-Pearson form.
        }

        A related message in our work is that we believe there should be a broader appreciation of $[0, 1]$-valued tests over $\{0, 1\}$-valued tests, which often seem shunned in the literature due to their classical connection to randomization.
        Indeed, $[0, 1]$-valued tests are much more amenable to optional continuation.
        Here, we personally prefer rescaling to the $[0, 1/\alpha]$ `evidence-scale' that gives rise to the e-value, for reasons listed in \citet{koning2024continuoustesting}.
        However, our work shows that anytime validity and optional continuation can be used on the traditional $[0, 1]$-scale, which is more familiar to most statisticians.

        \added{
        A final interesting finding is that for the purpose of optional continuation, our Theorem \ref{thm:pointwise_continuation} suggests that one should report the entire test family $(\phi^{\mathbb{P}})_{\mathbb{P} \in H}$ instead of the single test outcome $\phi^H := \inf_{\mathbb{P} \in H}$.
        This connects to calls to report confidence sets over test outcomes: in  \citet{koning2025optimal} it is shown that such a collection may be interpreted as a $[0,1]$-valued `fuzzy' generalization of a confidence set, and explained how they may be used in subsequent decision making (see also \citet{grunwald2023posterior, andrews2025certified, kiyani2025decision}).
        }
        
    \section{Acknowledgments}
        We sincerely thank Wouter Koolen for feedback on an early presentation of this work: he confirmed our suspicions that our results were far more general than we were able to show at the time.
        Moreover, we thank Simon Benhaiem, Lasse Fischer, Peter Gr\"unwald, Jesse Hemerik, Stan Koobs, Aaditya Ramdas and Dante de Roos for their comments.
        We also thank two anonymous referees for their valuable input.

    {
    \bibliographystyle{plainnat}
    \bibliography{Bibliography-MM-MC}

@article{kiyani2025decision,
  title={Decision theoretic foundations for conformal prediction: Optimal uncertainty quantification for risk-averse agents},
  author={Kiyani, Shayan and Pappas, George and Roth, Aaron and Hassani, Hamed},
  journal={arXiv preprint arXiv:2502.02561},
  year={2025}
}

@article{andrews2025certified,
  title={Certified Decisions},
  author={Andrews, Isaiah and Chen, Jiafeng},
  journal={arXiv preprint arXiv:2502.17830},
  year={2025}
}

@article{grunwald2023posterior,
  title={The e-posterior},
  author={Gr{\"u}nwald, Peter D},
  journal={Philosophical Transactions of the Royal Society A},
  volume={381},
  number={2247},
  pages={20220146},
  year={2023},
  publisher={The Royal Society}
}

@article{koning2025optimal,
  title={Fuzzy Prediction Sets: Conformal Prediction with E-values},
  author={Koning, Nick W and van Meer, Sam},
  journal={arXiv preprint arXiv:2509.13130},
  year={2025}
}

@book{Kallenberg2021foundations,
  author    = {Olav Kallenberg},
  title     = {Foundations of Modern Probability},
  edition   = {3},
  year      = {2021},
  publisher = {Springer},
  address   = {Cham},
  series    = {Probability Theory and Stochastic Modelling},
  volume    = {99},
  doi       = {10.1007/978-3-030-61871-1},
  isbn      = {978-3-030-61870-4},
  url       = {https://link.springer.com/book/10.1007/978-3-030-61871-1},
  note      = {eBook ISBN 978-3-030-61871-1; Softcover ISBN 978-3-030-61873-5}
}

@book{kuczma2009introduction,
  title={An introduction to the theory of functional equations and inequalities: Cauchy’s equation and Jensen’s inequality},
  author={Kuczma, Marek},
  year={2009},
  publisher={Springer}
}

@article{grunwald2025supermartingales,
  title={Supermartingales for one-sided tests: Sufficient monotone likelihood ratios are sufficient},
  author={Gr{\"u}nwald, Peter and Koolen, Wouter M},
  journal={Statistics \& Probability Letters},
  pages={110574},
  year={2025},
  publisher={Elsevier}
}

@article{muller2004general,
  title={A general statistical principle for changing a design any time during the course of a trial},
  author={M{\"u}ller, Hans-Helge and Sch{\"a}fer, Helmut},
  journal={Statistics in medicine},
  volume={23},
  number={16},
  pages={2497--2508},
  year={2004},
  publisher={Wiley Online Library}
}

@article{howard2020time,
  title={Time-uniform Chernoff bounds via nonnegative supermartingales},
  author={Howard, Steven R and Ramdas, Aaditya and McAuliffe, Jon and Sekhon, Jasjeet},
  journal={Probability Surveys},
  volume={17},
  pages={257--317},
  year={2020}
}

@article{lachin2005review,
author = {Lachin, John M.},
title = {A review of methods for futility stopping based on conditional power},
journal = {Statistics in Medicine},
volume = {24},
number = {18},
pages = {2747-2764},
doi = {https://doi.org/10.1002/sim.2151},
year = {2005}
}

@article{fischer2024improving,
  title={Improving the (approximate) sequential probability ratio test by avoiding overshoot},
  author={Fischer, Lasse and Ramdas, Aaditya},
  journal={arXiv preprint arXiv:2410.16076},
  year={2024}
}

@article{koning2024more,
    author = {Koning, N W and Hemerik, J},
    title = {More efficient exact group invariance testing: using a representative subgroup},
    journal = {Biometrika},
    volume = {111},
    number = {2},
    pages = {441-458},
    year = {2024},
    month = {06},
    issn = {1464-3510},
    doi = {10.1093/biomet/asad050}
}

@article{efron1969student,
  title={Student's t-test under symmetry conditions},
  author={Efron, Bradley},
  journal={Journal of the American Statistical Association},
  volume={64},
  number={328},
  pages={1278--1302},
  year={1969},
  publisher={Taylor \& Francis}
}

@misc{koning2024posthocalphahypothesistesting,
      title={Post-hoc $\alpha$ Hypothesis Testing and the Post-hoc $p$-value}, 
      author={Nick W. Koning},
      year={2024},
      eprint={2312.08040},
      archivePrefix={arXiv},
      primaryClass={math.ST},
      url={https://arxiv.org/abs/2312.08040}, 
}

@article{GordonLan01011982,
author = {K.K. Gordon Lan and Richard Simon and Max Halperin},
title = {Stochastically curtailed tests in long–term clinical trials},
journal = {Communications in Statistics. Part C: Sequential Analysis},
volume = {1},
number = {3},
pages = {207--219},
year = {1982},
publisher = {Taylor \& Francis},
doi = {10.1080/07474948208836014}
}

@misc{koning2024continuoustesting,
      title={Continuous Testing: Unifying Tests and E-values}, 
      author={Nick W. Koning},
      year={2025},
      eprint={2409.05654},
      archivePrefix={arXiv},
      primaryClass={math.ST},
      url={https://arxiv.org/abs/2409.05654}, 
}

@article{lehmann1949theory,
  title={On the theory of some non-parametric hypotheses},
  author={Lehmann, Eric L and Stein, Charles},
  journal={The Annals of Mathematical Statistics},
  volume={20},
  number={1},
  pages={28--45},
  year={1949},
  publisher={Institute of Mathematical Statistics}
}

@article{grunwald2024safe,
    author = {Gr\"unwald, Peter and de Heide, Rianne and Koolen, Wouter},
    title = {Safe testing},
    journal = {Journal of the Royal Statistical Society Series B: Statistical Methodology},
    volume = {86},
    number = {5},
    pages = {1091-1128},
    year = {2024},
    month = {03},
    issn = {1369-7412},
    doi = {10.1093/jrsssb/qkae011}
}

@article{perez2024e-statistics,
author = {Muriel Felipe P{\'e}rez-Ortiz and Tyron Lardy and Rianne de Heide and Peter D. Gr{\"u}nwald},
title = {{E-statistics, group invariance and anytime-valid testing}},
volume = {52},
journal = {The Annals of Statistics},
number = {4},
publisher = {Institute of Mathematical Statistics},
pages = {1410 -- 1432},
year = {2024},
doi = {10.1214/24-AOS2394},
URL = {https://doi.org/10.1214/24-AOS2394}
}

@article{ramdas2023game,
  title={Game-theoretic statistics and safe anytime-valid inference},
  author={Ramdas, Aaditya and Gr{\"u}nwald, Peter and Vovk, Vladimir and Shafer, Glenn},
  journal={Statistical Science},
  volume={38},
  number={4},
  pages={576--601},
  year={2023},
  publisher={Institute of Mathematical Statistics}
}

@article{howard2021time,
      author = {Steven R. Howard and Aaditya Ramdas and Jon McAuliffe and Jasjeet Sekhon},
      title = {{Time-uniform, nonparametric, nonasymptotic confidence sequences}},
      volume = {49},
      journal = {The Annals of Statistics},
      number = {2},
      publisher = {Institute of Mathematical Statistics},
      pages = {1055 -- 1080},
      year = {2021},
      doi = {10.1214/20-AOS1991}
}

@article{ramdas2022admissible,
      title={Admissible anytime-valid sequential inference must rely on nonnegative martingales}, 
      author={Aaditya Ramdas and Johannes Ruf and Martin Larsson and Wouter Koolen},
      year={2022},
      journal={arXiv preprint arXiv:2009.03167},
      eprint={2009.03167},
      archivePrefix={arXiv},
      primaryClass={math.ST}
}

@article{wasserman2020universal,
  title={Universal inference},
  author={Wasserman, Larry and Ramdas, Aaditya and Balakrishnan, Sivaraman},
  journal={Proceedings of the National Academy of Sciences},
  volume={117},
  number={29},
  pages={16880--16890},
  year={2020},
  publisher={National Acad Sciences}
}

@article{
grunwald2024beyond,
author = {Peter D. Grünwald },
title = {Beyond Neyman–Pearson: E-values enable hypothesis testing with a data-driven alpha},
journal = {Proceedings of the National Academy of Sciences},
volume = {121},
number = {39},
pages = {e2302098121},
year = {2024},
doi = {10.1073/pnas.2302098121}
}

@article{shafer2021testing,
  title={Testing by betting: A strategy for statistical and scientific communication},
  author={Shafer, Glenn},
  journal={Journal of the Royal Statistical Society Series A: Statistics in Society},
  volume={184},
  number={2},
  pages={407--431},
  year={2021},
  publisher={Oxford University Press}
}

@article{larsson2024numeraire,
  title={The numeraire e-variable and reverse information projection},
  author={Larsson, Martin and Ramdas, Aaditya and Ruf, Johannes},
  journal={The Annals of Statistics},
  volume={53},
  number={3},
  pages={1015--1043},
  year={2025},
  publisher={Institute of Mathematical Statistics}
}

@article{breiman1961optimal,
  title={Optimal gambling systems for favorable games},
  author={Breiman, Leo},
  journal={The Kelly Capital Growth Investment Criterion},
  pages={47--60},
  year={1961}
}

@article{posch2004conditional,
  title={Conditional rejection probabilities of student's t-test and design adaptations},
  author={Posch, Martin and Timmesfeld, Nina and K{\"o}nig, Franz and M{\"u}ller, Hans-Helge},
  journal={Biometrical Journal: Journal of Mathematical Methods in Biosciences},
  volume={46},
  number={4},
  pages={389--403},
  year={2004},
  publisher={Wiley Online Library}
}

@article{wang2024anytime,
  title={Anytime-valid t-tests and confidence sequences for Gaussian means with unknown variance},
  author={Wang, Hongjian and Ramdas, Aaditya},
  journal={Sequential Analysis},
  pages={1--55},
  year={2024},
  publisher={Taylor \& Francis}
}

@book{DoucetDeFreitasGordon2001,
  title={Sequential Monte Carlo Methods in Practice},
  editor={Doucet, Arnaud and de Freitas, Nando and Gordon, Neil J.},
  year={2001},
  publisher={Springer},
  address={New York},
}

@article{petrov1965probabilities,
  title={On the probabilities of large deviations for sums of independent random variables},
  author={Petrov, Valentin Vladimirovich},
  journal={Theory of Probability \& Its Applications},
  volume={10},
  number={2},
  pages={287--298},
  year={1965},
  publisher={SIAM}
}
    }
    \appendix  
    \section{\added{Counterexample for Proposition \ref{prp:condition_2}}} \label{app:counter}
        Suppose we sequentially observe $(X_1,X_2)$ taking values in $\{0,1\}\times\{0,1,2\}$ with probability mass
        \begin{align*}
             P_\theta(x_1,x_2)\propto \exp\{\theta\,k(x_1,x_2)\}, \qquad \theta\in \mathbb{R},
        \end{align*}
        where the function $k : \{0, 1\} \times \{0, 1, 2\} \to \{0, 1, 2, 3\}$ is given in the following table:
        \begin{align*}
            \begin{array}{c|ccc}
                k(x_1,x_2) & x_2 = 0 & x_2 = 1 & x_2 = 2\\ 
                \hline
                x_1 = 0 & 0 & 2 & 2 \\
                x_1 = 1 & 2 & 3 & 1
            \end{array}
        \end{align*}
        Define the statistics $T_1:=X_1$ and $T:=k(X_1,X_2)$.
        
        Trivially, $T_1$ is sufficient for $X_1$, and $T$ is sufficient for $(X_1,X_2)$ by the factorization theorem. Moreover, both admit monotone likelihood ratios: for $T_1$ the odds ratio
        \begin{align*}
            \frac{P_\theta(T_1=1)}{P_\theta(T_1=0)}=\frac{e^\theta+e^{2\theta}+e^{3\theta}}{1+2e^{2\theta}}
        \end{align*}
        is increasing in $\theta$, while for $T$, $P_\theta(T=s)/P_\theta(T=r)\propto e^{(s-r)\theta}$ is increasing in $\theta$ for $s>r$.  
        
        Nonetheless, stochastic ordering fails. At $\theta=\ln 2$, 
        \begin{align*}
        P(T\ge2\mid T_1=0)&=\frac{4+4}{1+4+4}=\tfrac{8}{9},\\ \qquad P(T\ge2\mid T_1=1)&=\frac{4+8}{4+8+2}=\tfrac{6}{7},
        \end{align*}
        hence $T$ is not stochastically non-decreasing in $T_1$.

    \section{Proofs}\label{appn:proofs}
        \subsection{Proof of Theorem \ref{thm:simple}}\label{app:simple}
             \begin{proof}
                We start with the first claim.
                By the law of iterated expectations,
                \begin{align*}
                    \mathbb{E}^{\mathbb{P}}[\phi_n \mid \mathcal{F}_{n-1}]
                        = \mathbb{E}^{\mathbb{P}}[\mathbb{E}^{\mathbb{P}}[\phi \mid \mathcal{F}_{n}] \mid \mathcal{F}_{n-1}]
                        = \mathbb{E}^{\mathbb{P}}[\phi \mid \mathcal{F}_{n-1}]
                        = \phi_{n-1},
                \end{align*}
                so that $(\phi_n)$ is a martingale with respect to $(\mathcal{F}_n)$.
                Since $\phi_n \in [0, 1]$, the sequence $(\phi_n)$ is uniformly integrable so that by Doob's optional stopping theorem for uniformly integrable martingales, we have that
                \begin{align}\label{eq:doob_martingale}
                    \mathbb{E}^{\mathbb{P}}[\phi_\tau] = \mathbb{E}^{\mathbb{P}}[\phi],
                \end{align}
                for every stopping time $\tau$.
                Hence, if $\phi$ is valid, we have $\sup_{\tau} \mathbb{E}^{\mathbb{P}}[\phi_\tau] \leq \alpha$.

                \added{The second claim also immediately follows from \eqref{eq:doob_martingale}, because if $\phi$ is not valid, then for every stopping time $\mathbb{E}^{\mathbb{P}}[\phi_\tau] = \mathbb{E}^{\mathbb{P}}[\phi] > \alpha$.}
    
                \added{For the final claim, suppose that $\phi$ is $\mathcal{F}_\tau$-measurable.
                Then, $\phi_\tau = \mathbb{E}^{\mathbb{P}}[\phi \mid \mathcal{F}_{\tau}] = \phi$.
                As a result, $\phi = \phi_\tau$.}
            \end{proof}

        \subsection{Proof of Proposition \ref{prp:validity_terminal_test}}\label{appn:validity_terminal_test}
            \begin{proof}
                If $(\phi_n)$ is anytime valid for $H$ then, because $\tau = \infty$ is a stopping time,
                \begin{align*}
                    \mathbb{E}^{\mathbb{P}}[\phi_\infty] 
                        \leq \sup_{\tau} \mathbb{E}^{\mathbb{P}}[\phi_\tau] \leq \alpha.
                \end{align*}
            \end{proof}

        \subsection{Proof of Proposition \ref{prp:admissible}}\label{appn:admissible}
            \begin{proof}
                Using the translation between e-values and tests described in Section \ref{sec:e-values}, we may apply the second claim of Theorem 18 in \citet{ramdas2022admissible} to obtain that $(\phi_n)$ is admissible if and only if it is a martingale.
                Now, as each $\phi_n$ is bounded in $[0, 1]$, $(\phi_n)$ is uniformly integrable.
                By Theorem 9.22 in \citet{Kallenberg2021foundations}, this implies $(\phi_n)$ is closable at $\sup n = \infty$.
                That is, $\phi_n = \mathbb{E}^{\mathbb{P}}[\phi_\infty \mid \mathcal{F}_n]$, $\mathbb{P}$-almost surely, for every $n$.
            \end{proof}

        \subsection{Proof of Proposition \ref{prp:e_to_essinf_e}}\label{appn:e_to_essinf_e}
            \begin{proof}
                We start with the left-to-right direction.
                Suppose that $\phi^H$ is valid for $H$.
                For every $\mathbb{P} \in H$, set $\phi^{\mathbb{P}} = \phi^H$.
                As $\phi^H$ is valid for $H$, it is valid for every $\mathbb{P} \in H$, which implies $\phi^{\mathbb{P}}$ is valid for $\mathbb{P}$.
                Moreover, $\essinf_{\mathbb{P} \in H} \phi^{\mathbb{P}} = \phi^H$.
        
                For the converse, suppose every $\phi^{\mathbb{P}}$ is valid for $\mathbb{P}$.
                Then,
                \begin{align*}
                    \sup_{\mathbb{P} \in H} \mathbb{E}^{\mathbb{P}}[\phi^H] \equiv \sup_{\mathbb{P} \in H} \mathbb{E}^{\mathbb{P}}[\essinf_{\mathbb{P}' \in H} \phi^{\mathbb{P}'}] \leq \sup_{\mathbb{P} \in H}  \mathbb{E}^{\mathbb{P}}[\phi^{\mathbb{P}}] \leq \alpha,
                \end{align*}
                so that $\phi^H$ is valid for $H$.
            \end{proof}
        
        \subsection{Proof of Theorem \ref{thm:composite}}\label{appn:composite}
            \begin{proof}
                For the first claim, we have
                \begin{align*}
                     \sup_{\tau \in \mathcal{T}} \sup_{\mathbb{P} \in H} \mathbb{E}^{\mathbb{P}}[\phi_\tau^H]
                        &= \sup_{\mathbb{P} \in H} \sup_{\tau \in \mathcal{T}} \mathbb{E}^{\mathbb{P}}[\phi_\tau^H]
                        = \sup_{\mathbb{P} \in H}  \sup_{\tau \in \mathcal{T}} \mathbb{E}^{\mathbb{P}}\left[\essinf_{\mathbb{P}' \in H}\mathbb{E}^{\mathbb{P}'}[\phi^{\mathbb{P}'} \mid \mathcal{F}_{\tau}]\right] \\
                        &\leq \sup_{\mathbb{P} \in H}  \sup_{\tau \in \mathcal{T}} \mathbb{E}^{\mathbb{P}}\left[\mathbb{E}^{\mathbb{P}}\left[\phi^{\mathbb{P}} \mid \mathcal{F}_{\tau}\right]\right]
                        =  \sup_{\mathbb{P} \in H} \mathbb{E}^{\mathbb{P}}[\phi^{\mathbb{P}}] \leq \alpha,
                \end{align*}
                \added{where the final equality follows from the law of iterated expectations, since $\mathcal{F}_\tau$ is a $\sigma$-algebra.}
        
                For the third claim, we assume that for every $\mathbb{P} \in H$, $\phi^{\mathbb{P}}$ is $\mathcal{F}_{\tau}$-measurable.
                As a consequence,
                \begin{align*}
                    \phi_\tau^H
                        &= \essinf_{\mathbb{P} \in H} \mathbb{E}^{\mathbb{P}}[\phi^{\mathbb{P}} \mid \mathcal{F}_{\tau}]
                        = \essinf_{\mathbb{P} \in H} \phi^{\mathbb{P}}.
                \end{align*}
        
                For the second claim, we assume that $\phi^H$ is not valid.
                This implies there exists some $\mathbb{P}^* \in H$ such that $\mathbb{E}^{\mathbb{P}^*}[\essinf_{\mathbb{P} \in H} \phi^{\mathbb{P}}] > \alpha$.
                Since each $\phi^{\mathbb{P}}$ is $\mathcal{F}_\infty$-measurable, applying the third claim with $\tau = \infty$ yields $\phi_\infty^H = \essinf_{\mathbb{P} \in H} \phi^{\mathbb{P}}$.
                Hence, $\mathbb{E}^{\mathbb{P}^*}[\phi_\infty^H] = \mathbb{E}^{\mathbb{P}^*}[\essinf_{\mathbb{P} \in H} \phi^{\mathbb{P}}] > \alpha$.
                As a consequence, there exists a stopping time $\tau$ such that $\phi_\tau^H$ is not valid, which implies $(\phi_n^H)$ is not anytime valid.
            \end{proof}

        \subsection{Proof of Proposition \ref{prp:condition_2}}\label{appn:condition_2}
        	\begin{proof}
                By the tower property, we have
                \begin{align*}
                    \phi_n
                        = \mathbb{E}^{\mathbb{P}_{\delta^0}}[\phi \mid \mathcal{F}_n]
                        = \mathbb{E}^{\mathbb{P}_{\delta^0}}[\phi \mid T_n]
                        = \mathbb{E}^{\mathbb{P}_{\delta^0}}[\mathbb{E}[\phi \mid T] \mid T_n].
                \end{align*}
                As $\phi$ is non-decreasing in $T$ and $T$ is stochastically non-decreasing in $T_n$, we have that $\phi_n$ is non-decreasing in $T_n$.
            \end{proof}

        \subsection{Proof of Proposition \ref{prp:seqt}}\label{appn:seqt}
        	\begin{proof}
        	    We verify the conditions of Theorem \ref{thm:mlr}, for $\phi_n = \phi_n'$.
        	    By scale-invariance of $T_1, \dots, T_n$, we need only consider $\sigma^2 = 1$ \citep{perez2024e-statistics}.
        	    For $\mu = 0$ and $\sigma^2 = 1$, $\phi_n$ is a non-negative martingale starting at 1 by Theorem \ref{thm:simple}.
        	    The Monotone Likelihood Ratio property holds in $\mu$ for $T_n$, for every $n$ \citep{grunwald2025supermartingales}.
        	    Hence, it remains to show Condition 2 of Theorem \ref{thm:mlr}.
        
        	    For this, it suffices to show that the two properties in Proposition \ref{prp:condition_2} hold.
        	    The first property is evident, as $\phi_N$ simply thresholds $T_N$.
        	    The second property follows from Lemma \ref{lem:increasing}.
        	\end{proof}

        \subsection{Proof of Lemma \ref{lem:increasing}}\label{app:proofseqt}
            In treatments of the $t$-test, it is common to rescale by the standard deviation of the data.
            We find it dramatically more convenient to instead map the data $X^n$ to the unit sphere by dividing by its norm $\|X^n\|_2$:
            \begin{align*}
                U^n := \frac{X^n}{\|X^n\|}.
            \end{align*}
            
            The $t$-test can be re-expressed as a Beta-test \citep{koning2024more}.
            In particular, let $B_n = \iota_n'U_n$ denote our beta-statistic, where $\iota_n := n^{-1/2}(1, \dots, 1)$.
            Then, the $t$-test is equivalent to 
            \begin{align*}
                \mathbb{I}\{B_n > \beta_{1-\alpha}^{N-n}\},
            \end{align*}
            where $\beta_{1-\alpha}^{N-n}$ is the $\alpha$-upper quantile of a Beta$\left(\frac{N-n-1}{2}, \frac{N-n-1}{2}\right)$-distribution on the interval $[-1, 1]$.
            The traditional expression of the $t$-test can be obtained through the strictly monotone transformation
            \begin{align}
                T_n \mapsto \frac{T_n}{\sqrt{T_n^2 + n - 1}} = B_n.
            \end{align}
    
            We will now prove the following lemma, which implies that $B_N$ is stochastically increasing in $B_n$, which in turn implies that $T_N$ is stochastically increasing in $T_n$ under $\mu = 0$ and $\sigma^2 > 0$, proving Lemma \ref{lem:increasing}.
            In fact, the result only uses that $U^n$ is uniform on the unit sphere, which holds if $X^n$ is rotationally invariant, not necessarily Gaussian.
    
            \begin{lem}\label{lem:conditional_beta}
                The conditional distribution of $B_N$ given $B_n = b_n$ can be represented as
                \begin{align*}
                    B_N \mid (B_n = b_n) \overset{d}{=} \sqrt{\frac{n}{N}} b_n \sqrt{W} + \sqrt{\frac{N-n}{N}}\widetilde{B}_{N-n} \sqrt{1-W},
                \end{align*}
                where $\widetilde{B}_{N-n} \sim \textnormal{Beta}\left(\frac{N-n-1}{2}, \frac{N-n-1}{2}\right)$ on the interval $[-1, 1]$, and $W \sim \textnormal{Beta}\left(\frac{n}{2}, \frac{N-n}{2}\right)$, independently.
            \end{lem}
            \begin{proof}
                To derive the conditional distribution of $B_N$ given $B_n$, it is convenient to take a brief detour through the Gaussian distribution so that we can use some well-known results on the Gaussian distribution.
    
                Let $Z^n$ denote a multivariate standard Gaussian random variable.
                Recall that a draw from $Z^n$ can be decomposed into the product of two independent random variables: $Z^n = \nu_n U^n$, where $\nu_n \sim \chi_n$ and $U^n$ is uniform on the sphere.
                Moreover, by the independence of the elements of $Z^N$, we can decompose it further into:
                \begin{align*}
                    Z^N
                        \overset{d}{=} (\nu_n U^n, \tilde{\nu}_{N- n} \widetilde{U}^{ N-n}),
                \end{align*}
                where independently, $\tilde{\nu}_{N - n}\sim \chi_{N-n}$ and $\widetilde{U}^{ N-n}$ is uniform on the sphere in dimension $(N-n)$.
    
                The conditional distribution of $B_N \mid B_n = b_n$ can then be expressed as
                \begin{align*}
                    B_N \mid (B_n = b_n)
                        &\overset{d}{=}\frac{\iota'Z^N}{||Z^N||}\Big|\frac{\iota'Z^n}{||Z^n||} = b_n \\
                        &\overset{d}{=}\frac{\frac{\sqrt{n}}{\sqrt{N}}b_n\nu_n+\frac{\sqrt{N-n}}{\sqrt{N}}\iota'Z_{N-n}}{\sqrt{\nu_n^2+||Z_{N-n}||^2}} \\
                        &\overset{d}{=} \frac{\frac{\sqrt{n}}{\sqrt{N}}b_n\nu_n+\frac{\sqrt{N-n}}{\sqrt{N}}\tilde{\nu}_{N-n}\iota'\widetilde{U}^{N-n}}{\sqrt{\nu_n^2+\tilde{\nu}_{N-n}^2}} \\
                        &\overset{d}{=} \frac{\frac{\sqrt{n}}{\sqrt{N}}b_n\nu_n+\frac{\sqrt{N-n}}{\sqrt{N}}\tilde{\nu}_{N-n}\widetilde{B}_{N-n}}{\sqrt{\nu_n^2+\tilde{\nu}_{N-n}^2}},
                \end{align*}
                where $\widetilde{B}_{N-n} \sim \textnormal{Beta}\left(\frac{N-n-1}{2}, \frac{N-n-1}{2}\right)$ on the interval $[-1, 1]$.
                
                This is an increasing function of $b_n$, so that we could here already immediately conclude that $B_N$ is stochastically increasing in $B_n$, so that $T_N$ is stochastically increasing in $T_n$.
                However, we continue to obtain the nicer expression.
    
                Next, recall that if $Y_1 \sim \chi_{k_1}^2$ and $Y_2 \sim\chi_{k_2}^2$, independently, then
                \begin{align*}
                    \frac{Y_1}{Y_1 + Y_2} \sim \textnormal{Beta}\left(\frac{k_1}{2}, \frac{k_2}{2}\right).
                \end{align*}
                As a consequence,
                \begin{align*}
                    B_N \mid (B_n = b_n)
                        &\overset{d}{=} c\sqrt{\frac{\nu_n^2}{\nu_n^2+\tilde{\nu}_{N-n}^2}}+d\sqrt{\frac{\tilde{\nu}_{N-n}^2}{\nu_n^2+\tilde{\nu}_{N-n}^2}} \\
                        &\overset{d}{=} c\sqrt{W}+d\sqrt{1-W},
                \end{align*}
                for $c = \sqrt{\frac{n}{N}}B_n$ and $d = \sqrt{\frac{N-n}{N}}\widetilde{B}_{N-n}$.
                Substituting $c$ and $d$ back in yields the desired expression.
            \end{proof}
            
        \subsubsection{Derivation sequential $t$-test}\label{app:deriv}
            We use the expression
            \begin{align*}
                B_N \mid (B_n = b_n) \overset{d}{=} \sqrt{\frac{n}{N}} b_n \sqrt{W} + \sqrt{\frac{N-n}{N}}\widetilde{B}_{N-n} \sqrt{1-W},
            \end{align*}
            derived in Lemma \ref{lem:conditional_beta} to obtain a nice expression for the tail probability
            \begin{align*}
                \phi_n' &= \mathbb{P}\Bigl[T_N > c_{\alpha_N}\mid T_n = t_n \Bigr]= \mathbb{P}\Bigl[B_N > \beta_{1-\alpha}\mid B_n = b_n \Bigr].
            \end{align*}
            
            The conditional probability given $W = w$ can be expressed as
            \begin{align*}
                F_{\widetilde{B}_{N-n}}\left(\frac{\sqrt{n} b_n \sqrt{w} - \sqrt{N} \beta_{1-\alpha}}{\sqrt{N-n}\sqrt{1-w}}\right),
            \end{align*}
            where $F_{\widetilde{B}_{N-n}}$ is the CDF of a Beta($\frac{N-1}{2},\frac{N-1}{2})$-distribution on [-1,1].
            Then integrating out $w$ yields
            \begin{align*}
                \phi'_n 
                    &=\int_0^1 F_{\widetilde{B}_{N-n}}\left(\frac{\sqrt{n} b_n \sqrt{w} - \sqrt{N} \beta_{1-\alpha}}{\sqrt{N-n}\sqrt{1-w}}\right) f(w)dw,
            \end{align*}
            where $f(w)$ is the density of a Beta$\left(\frac{n}{2},\frac{N-n}{2}\right)$ distribution.
            Recalling that $b_n = n^{-1/2}\sum_{i=1}^n x_i / \|x^n\|_2$ and some rewriting yields the expression in the main text.

        \subsection{Proof of Theorem \ref{thm:seq_significance_level}}\label{appn:seq_significance_level}
        	\begin{proof}
                We start with the first claim.
                For every $\mathbb{P}\in H$, we have
                \begin{align*}
                    \mathbb{E}^{\mathbb{P}}[\phi^*]
                        = \ \mathbb{E}^{\mathbb{P}}[\mathbb{E}^{\mathbb{P}}[\phi^* \mid \mathcal{F}_{\sigma}]]
                        \leq \mathbb{E}^{\mathbb{P}}[\phi_{\sigma}]
                        \leq \alpha,
                \end{align*}
                where the final inequality follows from the anytime validity of $(\phi_n)_{n \in \mathbb{N}}$.
                
                \added{
                We continue with the second claim.
                With a slight abuse of notation, we write $\phi_n^\dagger
                        =  \phi_n \mathbb{I}\{n \leq \sigma\} + \phi_n^* \mathbb{I}\{n > \sigma\}$.
                Fix an arbitrary stopping time $\tau$ adapted to $(\mathcal{F}_n)_{n \geq 0}$ and fix $\mathbb{P} \in H$.
                We then split the expectation of $\phi_\tau^\dagger$ into
                \begin{align*}
                    \mathbb{E}^{\mathbb{P}}[\phi_\tau^\dagger]
                        = \mathbb{E}^{\mathbb{P}}[\phi_\tau \mathbb{I}\{\tau \leq \sigma\}] + \mathbb{E}^{\mathbb{P}}[\phi_\tau^*\mathbb{I}\{\tau > \sigma\}].
                \end{align*}
                Consider the stopping time $\tau' = \tau \vee \sigma$ with respect to $(\mathcal{F}_n)_{n \geq \sigma}$.
                On the event $\{\tau > \sigma\}$ we have $\tau' = \tau$ and hence $\phi_\tau^* = \phi_{\tau'}^*$.
                Combining this with the $\mathcal{F}_\sigma$-measurability of $\{\tau > \sigma\}$ and the tower property gives
                \begin{align*}
                    \mathbb{E}^{\mathbb{P}}[\phi_\tau^*\mathbb{I}\{\tau > \sigma\}]
                        = \mathbb{E}^{\mathbb{P}}[\phi_{\tau'}^*\mathbb{I}\{\tau > \sigma\}]
                        = \mathbb{E}^{\mathbb{P}}[\mathbb{E}^{\mathbb{P}}[\phi_{\tau'}^* \mid \mathcal{F}_\sigma] \mathbb{I}\{\tau > \sigma\}].
                \end{align*}
                Applying the conditional anytime validity of $(\phi_n^*)$ to $\tau'$ yields $\mathbb{E}^{\mathbb{P}}[\phi_{\tau'}^* \mid \mathcal{F}_\sigma] \leq \phi_\sigma$, and hence
                \begin{align*}
                     \mathbb{E}^{\mathbb{P}}[\phi_\tau^*\mathbb{I}\{\tau > \sigma\}]
                        \leq \mathbb{E}^{\mathbb{P}}[\phi_\sigma\mathbb{I}\{\tau > \sigma\}].
                \end{align*}
                Now, using that $\phi_{\tau \wedge \sigma} = \phi_\tau \mathbb{I}\{\tau \leq \sigma\} + \phi_\sigma \mathbb{I}\{\tau > \sigma\}$, we have
                \begin{align*}
                    \mathbb{E}^{\mathbb{P}}[\phi_\tau^\dagger]
                        \leq \mathbb{E}^{\mathbb{P}}[\phi_\tau \mathbb{I}\{\tau \leq \sigma\}] + \mathbb{E}^{\mathbb{P}}[\phi_\sigma \mathbb{I}\{\tau > \sigma\}]
                        = \mathbb{E}^{\mathbb{P}}[\phi_{\tau \wedge \sigma}]
                        \leq \alpha,
                \end{align*}
                where the final inequality follows from the fact that $\tau \wedge \sigma$ is a stopping time and the anytime validity of $(\phi_n)$.}
            \end{proof}
        
        \subsection{Proof of Theorem \ref{thm:pointwise_continuation}}\label{appn:pointwise_continuation}
            \begin{proof}
                We have
                \begin{align*}
                    \sup_{\mathbb{P} \in H} \mathbb{E}^{\mathbb{P}}\left[\essinf_{\mathbb{P}' \in H} \phi_2^{\mathbb{P}'}\right]
                        &\leq \sup_{\mathbb{P} \in H} \mathbb{E}^{\mathbb{P}}\left[\phi_2^{\mathbb{P}}\right]
                        = \sup_{\mathbb{P} \in H} \mathbb{E}^{\mathbb{P}}\left[\mathbb{E}^{\mathbb{P}}[\phi_2^{\mathbb{P}} \mid \mathcal{F}_1]\right]
                        \leq \sup_{\mathbb{P} \in H} \mathbb{E}^{\mathbb{P}}\left[\phi_1^{\mathbb{P}}\right]
                        \leq \alpha,
                \end{align*}
                where the second inequality follows from the conditional validity of $\phi_2^{\mathbb{P}}$ at level $\phi_1^{\mathbb{P}}$ and the final inequality from the validity of every $\phi_1^{\mathbb{P}}$.
            \end{proof}

        \added{
        \subsection{Proof of Theorem \ref{thm:log_invariance}}\label{appn:log_invariance}
        	\begin{proof}
        	    By log-utility-optimality, $\varepsilon = d\mathbb{Q} / d\mathbb{P}$, so that it remains to show that $\varepsilon_n = d\mathbb{Q}_{\mathcal{F}_n} / d\mathbb{P}_{\mathcal{F}_n}$, which is well defined as $\mathbb{P} \gg \mathbb{Q}$ implies $\mathbb{P}_{\mathcal{F}_n} \gg \mathbb{Q}_{\mathcal{F}_n}$.
        	    For every $F_n \in \mathcal{F}_n$, we have 
        	    \begin{align*}
        	        \mathbb{Q}(F_n)
        	            &= \mathbb{E}^{\mathbb{Q}}[1_{F_n}]
        	            = \mathbb{E}^{\mathbb{P}}[\varepsilon 1_{F_n}]
        	            = \mathbb{E}^{\mathbb{P}}[\mathbb{E}^{\mathbb{P}}[\varepsilon 1_{F_n} \mid \mathcal{F}_n]]
        	            = \mathbb{E}^{\mathbb{P}}[1_{F_n}\mathbb{E}^{\mathbb{P}}[\varepsilon \mid \mathcal{F}_n]]
        	            = \mathbb{E}^{\mathbb{P}}[1_{F_n} \varepsilon_n].
        	    \end{align*}
        	    Since $\mathbb{P}_{\mathcal{F}_n}(F_n) = \mathbb{P}(F_n)$ and $\mathbb{Q}_{\mathcal{F}_n}(F_n) = \mathbb{Q}(F_n)$ for $F_n \in \mathcal{F}_n$, this can be written as
        	    \begin{align*}
        	        \mathbb{Q}_{\mathcal{F}_n}(F_n)
        	            = \mathbb{E}^{\mathbb{Q}_{\mathcal{F}_n}}[1_{F_n}]
        	            = \mathbb{E}^{\mathbb{P}_{\mathcal{F}_n}}[1_{F_n} \varepsilon_n],
        	    \end{align*}
        	    for every $F_n \in \mathcal{F}_n$.
        	    Hence, the $\mathcal{F}_n$-measurable function $\varepsilon_n$ satisfies the defining property of the Radon-Nikodym derivative of $\mathbb{Q}_{\mathcal{F}_n}$ with respect to $\mathbb{P}_{\mathcal{F}_n}$, so that $\varepsilon_n = d\mathbb{Q}_{\mathcal{F}_n} / d\mathbb{P}_{\mathcal{F}_n}$, $\mathbb{P}$-almost surely.
        	\end{proof}
        }

        \added{
        \subsection{Proof of Proposition \ref{prp:define_log}}\label{appn:define_log}
        	\begin{proof}
        	    Let us write $\textnormal{LR}=d\mathbb{Q}/d\mathbb{P}$.
        	    Under the provided conditions on $U$ and the optimality of $\varepsilon$,
        	    \begin{align*}
        	        \varepsilon = (U')^{-1}\left(\lambda / \textnormal{LR}\right),
        	    \end{align*}
        	    for some $\lambda > 0$, since the constraint $\mathbb{E}^{\mathbb{P}}[\varepsilon] = 1$ must be binding here; else $\varepsilon$ could be strictly improved.
        	    Hence, $\varepsilon = g_\lambda(\textnormal{LR})$ for $g_\lambda(x) = (U')^{-1}(\lambda / x)$.
        	    Since $U'$ is continuous and strictly decreasing, $g_\lambda$ is continuous on $(0, \infty)$.
        
        	    By definition we have $\varepsilon_n := \mathbb{E}^{\mathbb{P}}[\varepsilon \mid \mathcal{F}_n] = \mathbb{E}^{\mathbb{P}}[g_\lambda(\textnormal{LR}) \mid \mathcal{F}_n]$.
        	    Moreover, by assumption $\varepsilon_n$ is $U$-optimal when restricted to $\mathcal{F}_n$, and so $\varepsilon_n = g_{\lambda_n}(\textnormal{LR}_n)$ for some $\lambda_n > 0$, where $\textnormal{LR}_n := d\mathbb{Q}_{\mathcal{F}_n}/d\mathbb{P}_{\mathcal{F}_n}$.
        	    
        	    By Theorem \ref{thm:log_invariance} we have $\textnormal{LR}_n = \mathbb{E}^{\mathbb{P}}[\textnormal{LR}\mid\mathcal{F}_n]$.
        	    Hence, for $\varepsilon_n$ to be $U$-optimal we require
        	    \begin{align}\label{eq:jensen}
        	        \mathbb{E}^{\mathbb{P}}[g_\lambda(\textnormal{LR})\mid\mathcal{F}_n]
        	            = g_{\lambda_n}\left(\mathbb{E}^{\mathbb{P}}[\textnormal{LR}\mid\mathcal{F}_n]\right).
        	    \end{align}
        
        	    We now use the fact that we assumed this holds for any $\mathbb{P}\gg\mathbb{Q}$ and underlying measure space, to construct a setting in which we may apply Theorem 13.2.2 from \citet{kuczma2009introduction}.
        	    In particular, let $x, y > 0$ and consider an underlying $\sigma$-algebra $\mathcal{F}$ such that there exist disjoint
        	    $A, B, C\in \mathcal{F}$ with $\mathbb{P}(A) = \mathbb{P}(B) > 0$ and $C = (A \cup B)^c$.
        	    Define
        	    $\textnormal{LR} = x 1_A + y 1_B + z 1_C$, where $z > 0$ is chosen so that
        	    $\mathbb{E}^{\mathbb{P}}[\textnormal{LR}] = 1$ (which is possible since we may choose $\mathbb{P}(A) = \mathbb{P}(B) > 0$ sufficiently small).
        	    Furthermore, define $\mathbb{Q}$ through $\textnormal{LR} = d\mathbb{Q}/d\mathbb{P}$.
        	    Finally, take $\mathcal{F}_n = \sigma(A\cup B, C)$.
        
        	    On $C$ we have $\mathbb{E}^{\mathbb{P}}[\textnormal{LR} \mid \mathcal{F}_n] = z$ and $\mathbb{E}^{\mathbb{P}}[g_\lambda(\textnormal{LR}) \mid \mathcal{F}_n] = g_\lambda(z)$, so that \eqref{eq:jensen} yields $g_\lambda(z) = g_{\lambda_n}(z)$.
        	    Since $\lambda \mapsto g_\lambda(z) = (U')^{-1}(\lambda/z)$ is strictly monotone, this implies $\lambda_n = \lambda$.
        	    
        	    Then, on $A \cup B$, we have both
        	    \begin{align*}
        	        \mathbb{E}^{\mathbb{P}}[\textnormal{LR}\mid\mathcal{F}_n]
        	            = \frac{x+y}{2}
        	    \end{align*}
        	    and
        	    \begin{align*}
        	        \mathbb{E}^{\mathbb{P}}[g_\lambda(\textnormal{LR})\mid\mathcal{F}_n]
        	            = \frac{g_\lambda(x) + g_\lambda(y)}{2}.
        	    \end{align*}
        	    Hence, since $\lambda = \lambda_n$, \eqref{eq:jensen} requires $g_\lambda(\frac{x+y}{2}) = (g_\lambda(x) + g_\lambda(y)) / 2$ for every $x, y > 0$.
        
        	    Now, by Theorem 13.2.2 from \citet{kuczma2009introduction} and the continuity of $g_\lambda$, we obtain that $g_\lambda$ is affine.
        	    Hence $(U')^{-1}(\lambda/x) = ax + b$, which implies $U'(u) = \lambda a / (u - b)$, and so $U(u) = c \log(u - b) + d$, where $a, b, c$ and $d$ are constants.
        	    Since $U$ is strictly increasing, $a > 0$ and $U'(u) > 0$ on $(0,\infty)$, hence $b \leq 0$.
        	    Finally, $\lim_{x\to 0} U'(x) = \infty$ forces $b = 0$, so that $U = \log$ up to affine transformations.
        	\end{proof}
        }
        
        \subsection{Proof of Theorem \ref{thm:LR_sequential_z-test}} \label{app:gaussLR}
            \begin{proof}
                We write $S_n := \log \mathrm{LR}_n$ and $D := D(\mathbb{Q}\|\mathbb{P})$.
                By the atomless assumption, the NP-test is of the form
                \begin{align*}
                    \phi_{N,\alpha_N} = \mathbb{I}\{S_N \geq t_N\},
                \end{align*}
                where $t_N$ is a unique scalar that ensures
                \begin{align}
                    \mathbb{P}(S_N \geq t_N) = \exp\{-N D\} \equiv \alpha_N. \label{eq:tnimplicit}
                \end{align}
                Our proof strategy is to use a large deviations theorem by \cite{petrov1965probabilities} to derive an asymptotic expansion for $t_N$, and then use this to derive an expression for $\phi'_{n,\alpha_N}$.
            
                \subsubsection*{Step 1: deriving an expression for $t_N$}
                In this step, we approximate the LHS of \eqref{eq:tnimplicit} through Theorem 2 of \cite{petrov1965probabilities} and use this to invert an expression for $t_N$.
                The conditions for this theorem are satisfied: the non-lattice condition is assumed, $D_{1+\varepsilon}(\mathbb{Q}\|\mathbb{P}) < \infty$ implies the moment-generating function is finite on $[0,1+\epsilon]$, and, by Cramér's theorem, $\mathbb P(S_N \ge t_N)= \exp\{-N D\}$ implies $t_N/N \to D$.
            
               Since $t_N/N\to D$, there exists some deterministic sequence $(b_N)$ such that
                \begin{equation*}\label{eq:tN_decomp}
                    t_N = ND + Nb_N, \quad \text{and $b_N \to 0$}.
                \end{equation*}
                We now apply Theorem~2 of \cite{petrov1965probabilities} with (in their notation) $\delta(N) := |b_N|$, $B:=1+\epsilon$, $c := D$ and $h: = 1$,
                \begin{align*}\label{eq:petrov_tN}
                    \mathbb{P}(S_N \geq t_N) &=  \mathbb{P}(S_N \geq ND+Nb_N)\\ &= \frac{1+o(1)}{\sigma(1)\sqrt{2\pi N}} \exp\left\{-ND - Nb_N - \frac{Nb_N^2}{2\sigma^2(1)}(1 + O(|b_N|))\right\},
                \end{align*}
                where $\sigma^2(1)= \text{Var}_\mathbb{Q}(\log \LR_1)$ is guaranteed to exist by Lemma 1 in \cite{petrov1965probabilities}.
            
              Now that we have an expression for the LHS of \eqref{eq:tnimplicit}, we solve for $t_N$. 
              Plugging this into $\mathbb{P}(S_N \geq t_N) = e^{-ND}$, taking logarithms and rearranging the terms yields:
                \begin{equation}\label{eq:balance}
                    Nb_N + \frac{Nb_N^2}{2\sigma^2(1)}(1 + O(|b_N|)) = -\frac{1}{2}\log N - \log(\sigma(1)\sqrt{2\pi}) + o(1).
                \end{equation}
                As this is an equation, both the LHS and the RHS should be of the same asymptotic order.
                The single highest asymptotic order on the RHS is $O(\log N)$.
                The first term on the LHS is $O(Nb_N)$, whereas the second term is $O(Nb_N^2)$. 
                Since $b_N \to 0$, the first term is of the single highest order.
                Hence, we must have $Nb_N = O(\log N)$, and therefore $Nb_N^2 = O((\log N)^2/N) = o(1)$.
                Thus, \eqref{eq:balance} simplifies to
                \begin{align*}
                    Nb_N = -\frac{1}{2}\log N - \log(\sigma(1)\sqrt{2\pi}) + o(1).
                \end{align*}
                Using that $t_N = ND + Nb_N$, we obtain
                \begin{equation}\label{eq:tN_asymp}
                    t_N = ND - \frac{1}{2}\log N - \log(\sigma(1)\sqrt{2\pi}) + o(1).
                \end{equation}
                \medskip
                \medskip
                \noindent\textbf{Step 2: Conditional probability.}
            
                Fix $n \in \mathbb{N}$ and $s \in \mathbb{R}$. By the i.i.d. assumption,
                \begin{align*}
                    \mathbb{P}(S_N \geq t_N \mid S_n = s) = \mathbb{P}(S_{N-n} \geq t_N - s).
                \end{align*}
                By the same arguments as in step 1, we know that $t_N/(N-n)\to D$.
                Since $s$ is fixed, this also implies
                $(t_N-s)/(N-n)\to D$.
                Hence there exists a deterministic sequence $(c_N)$ such that 
                \begin{align}
                t_N-s = (N-n)D + (N-n)c_N,\quad\text{and $c_N\to 0$}.
                    \label{eq:cn}
                \end{align}
               
                Applying Theorem~2 of \citet{petrov1965probabilities} to $S_{N-n}$ with $\delta(N) := |c_N|$, $B:=1+\epsilon$, $c := D$ and $h: = 1$,
                \begin{align}
                    \mathbb{P}(S_{N-n} \geq t_N-s) &= \frac{1+o(1)}{\sigma(1)\sqrt{2\pi(N-n)}} \\&\times \exp\left\{-(N-n)D - (N-n)c_N - \frac{(N-n)c_N^2}{2\sigma^2(1)}(1 + O(|c_N|))\right\}. \label{eq:petrov2}
                \end{align}
                We proceed by studying the terms $(N-n)c_N$ and $(N-n)c^2_N$.
                First using the definition of $c_N$ in \eqref{eq:cn} and then plugging in $t_N$ derived in \eqref{eq:tN_asymp},
                \begin{align*}
                    (N-n)c_N \equiv t_N - s - (N-n)D = nD - s - \frac{1}{2}\log N - \log(\sigma(1)\sqrt{2\pi}) + o(1).
                \end{align*}

                Since $(N-n)c_N= O(\log N)$, we have that $(N-n)c_N^2=O(\log N^2/N)$, which is $o(1)$. 
                As $c_N =o(1)$, it also holds that the term $(N-n)c_N^2(1+O(|c_N|))$ in \eqref{eq:petrov2} is $o(1)$.
            
                Plugging this into \eqref{eq:petrov2},
                \begin{align*}
                    \mathbb{P}(S_{N-n} \geq t_N - s) &= \frac{1+o(1)}{\sigma(1)\sqrt{2\pi(N-n)}} \exp\left\{-ND + s + \frac{1}{2}\log N + \log(\sigma(1)\sqrt{2\pi})+o(1)\right\} \\
                    &= (1+o(1)).\sqrt{\frac{N}{N-n}} \exp \{-ND + s+o(1)\}\\
                    &= (1+o(1))\exp\{s-ND\}
                \end{align*}
            
                Since $\alpha_N = \exp\{-ND\}$, for each fixed $s \in \mathbb{R}$:
                \begin{equation}\label{eq:limit_fixed_s}
                    \lim_{N \to \infty} \frac{\mathbb{P}(S_{N} \geq t_N|S_n= s)}{\alpha_N} = \exp\{s\}.
                \end{equation}
            
                \medskip
                \noindent\textbf{Step 3: Conclusion.}
                Since our result holds for any $s \in \mathbb{R}$, and $S_n$ is $\mathbb{P}$-a.s.\ finite,
                \begin{align*}
                    \lim_{N \to \infty} \frac{\phi'_{n,\alpha_N}}{\alpha_N} = \exp\{S_n\} = \mathrm{LR}_n \qquad \mathbb{P}\text{-a.s.}
                \end{align*}

                for each fixed $n \in \mathbb{N}$.
            \end{proof}
\end{document}